\documentclass[12pt,magyar,english]{article}
\usepackage[T1]{fontenc}
\usepackage[latin2,latin9]{inputenc}
\usepackage{geometry}
\usepackage{amsmath}
\usepackage{amsthm}
\usepackage{amssymb}
\usepackage{graphicx}

\makeatletter
\theoremstyle{plain}
\newtheorem{thm}{\protect\theoremname}[section]
  \theoremstyle{plain}
  \newtheorem{cor}[thm]{\protect\corollaryname}
  \theoremstyle{plain}
  \newtheorem{prop}[thm]{\protect\propositionname}
  \theoremstyle{remark}
  \newtheorem{rem}[thm]{\protect\remarkname}
  \theoremstyle{plain}
  \newtheorem{lem}[thm]{\protect\lemmaname}
  \theoremstyle{remark}
  \newtheorem{notation}[thm]{\protect\notationname}
  \theoremstyle{definition}
  \newtheorem{example}[thm]{\protect\examplename}
  \theoremstyle{plain}
  \newtheorem{question}[thm]{\protect\questionname}

\makeatother

\usepackage{babel}
  \addto\captionsenglish{\renewcommand{\corollaryname}{Corollary}}
  \addto\captionsenglish{\renewcommand{\examplename}{Example}}
  \addto\captionsenglish{\renewcommand{\lemmaname}{Lemma}}
  \addto\captionsenglish{\renewcommand{\notationname}{Notation}}
  \addto\captionsenglish{\renewcommand{\propositionname}{Proposition}}
  \addto\captionsenglish{\renewcommand{\questionname}{Question}}
  \addto\captionsenglish{\renewcommand{\remarkname}{Remark}}
  \addto\captionsenglish{\renewcommand{\theoremname}{Theorem}}
  \addto\captionsmagyar{\renewcommand{\corollaryname}{Következmény}}
  \addto\captionsmagyar{\renewcommand{\examplename}{Példa}}
  \addto\captionsmagyar{\renewcommand{\lemmaname}{Segédtétel}}
  \addto\captionsmagyar{\renewcommand{\notationname}{Jelölés}}
  \addto\captionsmagyar{\renewcommand{\propositionname}{Állítás}}
  \addto\captionsmagyar{\renewcommand{\questionname}{Kérdés}}
  \addto\captionsmagyar{\renewcommand{\remarkname}{Észrevétel}}
  \addto\captionsmagyar{\renewcommand{\theoremname}{Tétel}}
  \providecommand{\corollaryname}{Corollary}
  \providecommand{\examplename}{Example}
  \providecommand{\lemmaname}{Lemma}
  \providecommand{\notationname}{Notation}
  \providecommand{\propositionname}{Proposition}
  \providecommand{\questionname}{Question}
  \providecommand{\remarkname}{Remark}
\providecommand{\theoremname}{Theorem}

\begin{document}
\date{}

\title{Projections of self-similar sets with no separation condition}

\author{\'Abel Farkas}
\maketitle
\begin{abstract}
We investigate how the Hausdorff dimension and measure of a self-similar
set $K\subseteq\mathbb{R}^{d}$ behave under linear images. This depends
on the nature of the group $\mathcal{T}$ generated by the orthogonal
parts of the defining maps of $K$. We show that if $\mathcal{T}$
is finite then every linear image of $K$ is a graph directed attractor
and there exists at least one projection of $K$ such that the dimension
drops under the image of the projection. In general, with no restrictions
on $\mathcal{T}$ we establish that $\mathcal{H}^{t}\left(L\circ O(K)\right)=\mathcal{H}^{t}\left(L(K)\right)$
for every element $O$ of the closure of $\mathcal{T}$, where $L$
is a linear map and $t=\dim_{H}K$. We also prove that for disjoint
subsets $A$ and \textbf{$B$} of $K$ we have that $\mathcal{H}^{t}\left(L(A)\cap L(B)\right)=0$.
Hochman and Shmerkin showed that if $\mathcal{T}$ is dense in $SO(d,\mathbb{R})$
and the strong separation condition is satisfied then $\dim_{H}\left(g(K)\right)=\min\left\{ \dim_{H}K,l\right\} $
where $g$ is a continuously differentiable map of rank $l$. We deduce
the same result without any separation condition and we generalize
a result of Ero$\breve{\mathrm{g}}$lu by obtaining that $\mathcal{H}^{t}(g(K))=0$. 
\end{abstract}

\section{\label{sec:Introduction}Introduction}

\subsection{Overview}

Studying the Hausdorff dimension and measure of orthogonal projections
and linear images of sets has a long history. The most fundamental
result is that for an analytic subset $K$ of $\mathbb{R}^{d}$
\[
\dim_{H}\Pi_{M}(K)=\min\left\{ l,\dim_{H}(K)\right\} 
\]
for almost all $l$-dimensional subspaces $M$, where $\dim_{H}$
denotes the Hausdorff dimension and $\Pi_{M}:\mathbb{R}^{d}\longrightarrow M$
denotes the orthogonal projection onto $M$. If $\dim_{H}(K)>l$ then
\[
\mathcal{H}^{l}(\Pi_{M}(K))>0
\]
for almost all $l$-dimensional subspaces $M$, where $\mathcal{H}^{s}$
denotes the $s$-dimensional Hausdorff measure. These were proved
in the case $d=2$, $l=1$ by Marstrand \cite{Marstrand_paper}, and
generalized to higher dimensions by Mattila \cite{Mattila-H-dim of projections of sets}.
We call a set $K$ an $s$-set if $0<\mathcal{H}^{s}(K)<\infty$.
If $l$ is an integer than we call an $l$-set $K$ \textsl{irregular}
if $\mathcal{H}^{l}(K\cap M)=0$ for every differentiable $l$-manifold
$M$. It was shown by Besicovitch \cite{Besicovitch-Proj THM} in
the planar case and by Federer \cite{Federer-Proj THM} in the higher
dimensional case that for an $l$-set $K$ where $l$ is an integer
\[
\mathcal{H}^{l}(\Pi_{M}(K))=0
\]
for almost all $l$-dimensional subspaces $M$ if and only if $K$
is irregular. If $K$ is not irregular then $\mathcal{H}^{l}(\Pi_{M}(K))>0$
for almost all $l$-dimensional subspaces $M$.

While the results above provide information about generic projections
they do not give any information about an individual projection or
linear image of the set. There are examples that show that the `exceptional
set' for which the conclusions do not hold can be `big' \cite{Marstrand_paper}.
Analyzing the image of a set under a particular linear map is more
difficult even in simple cases, see for example Kenyon \cite{Kenyon-sierpinski}
and Hochman \cite[Theorem 1.6]{Hochman-overlaps} who consider the
$1$-dimensional Sierpinski gasket. Hence we restrict the attention
to a certain family of sets, namely we assume $K$ to be a self-similar
set.

While studying self-similar sets the `open set condition' is a convenient
assumption that makes the proofs significantly simpler. That is why
the case when the open set condition is satisfied is quite well-understood
but we know much less in the general situation when no separation
condition is assumed. The results in this paper include this general
situation. Recent results of Hochman were major breakthrough in studying
overlapping self-similar sets. A folklore conjecture is that for a
self-similar set $K$ on the line $\dim_{H}K<\min\left\{ 1,s\right\} $
if and only if exact overlapping occurs among the cylinder sets where
$s$ denotes the similarity dimension of $K$. Hochman \cite[Theorem 1.5]{Hochman-overlaps}
proves this conjecture when only algebraic parameters occur in the
defining maps of $K$. In Example \ref{Ex: nonOSC-line} we provide
a self-similar set $\widehat{K}\subseteq\mathbb{R}$ such that after
deleting any number of exact overlaps from the defining maps of $\widehat{K}$
we still have exact overlaps and hence the similarity dimension never
realises the Hausdorff dimension of the set even if changing the defining
maps.

It is easy to see that if $K$ is a self-similar set with all the
defining maps homotheties then every linear image of $K$ is itself
a self-similar set. It was asked by Mattila \cite[Problem 2]{Mattila-survey}
in the planar case `what can be said about the measures $\mathcal{H}^{t}(\Pi_{M}(K))$
if $t=\dim_{H}(K)<1$ and the defining maps contain rotations?'. Ero$\breve{\mathrm{g}}$lu
\cite{Eroglu-dense set of rotations} showed that if the open set
condition is satisfied and the orthogonal part of one of the defining
maps is a rotation of infinite order then $\mathcal{H}^{t}(\Pi_{M}(K))=0$
for every line $M$. We generalize this result to higher dimensions
and for continuously differentiable maps in place of projections without
assuming any separation condition. We obtain results on the structure
of linear images of $K$ if the transformation group generated by
the orthogonal parts of the defining maps is of finite order. We show
that linear images of such self-similar sets are graph directed attractors.
We establish an invariance result concerning the Hausdorff measure
of the linear images of $K$ in the general case with no restrictions
on the orthogonal transformation group. As a consequence of this we
conclude that for every linear map into another Euclidean space $L:\mathbb{R}^{d}\longrightarrow\mathbb{R}^{d_{2}}$
where $d_{2}$ is an arbitrary natural number and for disjoint subsets
$A$ and \textbf{$B$} of $K$ we have that $\mathcal{H}^{t}\left(L(A)\cap L(B)\right)=0$
even if no separation condition is satisfied. In particular, projection
of disjoint parts of $K$ are almost disjoint. 

Peres and Shmerkin \cite[Thorem 5]{Peres-Schmerkin Resonance between Cantor sets}
showed that if the orthogonal part of one of the defining maps is
a rotation of infinite order then
\[
\dim_{H}\Pi_{M}(K)=\min\left\{ 1,\dim_{H}(K)\right\} 
\]
for every line $M$. Very recently Hochman and Shmerkin \cite[Corollary 1.7]{Hochman-Schmerkin-local entropy}
generalized this to higher dimensions for continuously differentiable
maps in the strong separation condition case. Using their result and
a dimension approximation method we deduce the same conclusion without
any separation condition. On the other hand, we show that if the orthogonal
transformation group generated by the orthogonal parts of the defining
maps is of finite order then there exists a projection of $K$ such
that the dimension drops under the image of the projection.

\subsection{Definitions and Notations}

A \textit{self-similar iterated function system} (SS-IFS) in $\mathbb{R}^{d}$
is a finite collection of maps $\left\{ S_{i}\right\} _{i=1}^{m}$
from $\mathbb{R}^{d}$ to $\mathbb{R}^{d}$ such that all the $S_{i}$
are contracting similarities. The \textit{attractor} of the SS-IFS
is the unique nonempty compact set $K$ such that $K=\bigcup_{i=1}^{m}S_{i}(K)$.
The attractor of an SS-IFS is called a \textit{self-similar set}.

We say that the SS-IFS $\left\{ S_{i}\right\} _{i=1}^{m}$ satisfies
the \textit{strong separation condition} (SSC) if $\bigcup_{i=1}^{m}S_{i}(K)$
is a disjoint union. We say that the SS-IFS $\left\{ S_{i}\right\} _{i=1}^{m}$
satisfies the \textit{open set condition} (OSC) if there exists a
nonempty open set $U\subseteq\mathbb{R}^{d}$ such that
\[
\bigcup_{i=1}^{m}S_{i}(U)\subseteq U
\]
and the union is disjoint. It is easy to see that SSC implies OSC.

Let $\left\{ S_{i}\right\} _{i=1}^{m}$ be an SS-IFS. Then every $S_{i}$
can be uniquely decomposed as
\begin{equation}
S_{i}(x)=r_{i}T_{i}(x)+v_{i}\label{eq:S_ideconposition}
\end{equation}
for all $x\in\mathbb{R}^{d}$, where $0<r_{i}<1$, $T_{i}$ is an
orthogonal transformation and $v_{i}\in\mathbb{R}^{d}$ is a translation,
for all indices $i$. The unique solution $s$ of the equation
\begin{equation}
\sum_{i=1}^{m}r_{i}^{s}=1\label{eq: S-dim sum}
\end{equation}
is called the \textit{similarity dimension} of the SS-IFS. It is well-known
that if the SS-IFS satisfies the OSC then $0<\mathcal{H}^{s}(K)<\infty$.
Let $\mathcal{T}$ denote the group generated by the orthogonal transformations
$\left\{ T_{i}\right\} _{i=1}^{m}$. We call $\mathcal{T}$ the \textit{transformation
group} of the SS-IFS.

We denote the set $\left\{ 1,2,\ldots,m\right\} $ by $\mathcal{I}$.
Let $\boldsymbol{\mathbf{i}}=(i_{1},\ldots,i_{k})\in\mathcal{I}^{k}$
i.e. a $k$-tuple of indices. Then we write $S_{\boldsymbol{\mathbf{i}}}=S_{i_{1}}\circ\ldots\circ S_{i_{k}}$
and $K_{\boldsymbol{\mathbf{i}}}=S_{\boldsymbol{\mathbf{i}}}(K)$.
Since the similarities are decomposed as in (\ref{eq:S_ideconposition})
we write $r_{\boldsymbol{\mathbf{i}}}=r_{i_{1}}\cdot\ldots\cdot r_{i_{k}}$
and $T_{\boldsymbol{\mathbf{i}}}=T_{i_{1}}\circ\ldots\circ T_{i_{k}}$.
For an overview of the theory of self-similar sets see, for example,
\cite{Falconer1,Falconer Techniques,Hutchinson,Mattilakonyv,Schief OSC}.

We would like to avoid the singular non-interesting case, when $K$
is a single point, which occurs if and only if every $S_{i}$ has
the same fixed point. Hence we make the global assumption throughout
the whole paper that $K$ contains at least two points. This implies
that there are at least two maps in the SS-IFS, i.e. $m>1$. Hence
the similarity dimension of the SS-IFS is strictly positive. It is
relevant for us that the assumption that $K$ contains at least two
points also implies that $\dim_{H}K>0$ even with no separation condition.

Let $G\left(\mathcal{V},\mathcal{E}\right)$ be a directed graph,
where $\mathcal{V}=\left\{ 1,2,\ldots,q\right\} $ is the set of vertices
and $\mathcal{E}$ is the finite set of directed edges such that for
each $i\in\mathcal{V}$ there exists $e\in\mathcal{E}$ starting from
$i$. Let $\mathcal{E}_{i,j}$ denote the set of edges from vertex
$i$ to vertex $j$ and $\mathcal{E}_{i,j}^{k}$ denote the set of
sequences of $k$ edges $\left(e_{1},\ldots,e_{k}\right)$ which form
a directed path from vertex $i$ to vertex $j$. A \textit{graph directed
iterated function system} (GD-IFS) in $\mathbb{R}^{d}$ is a finite
collection of maps $\left\{ S_{e}:e\in\mathcal{E}\right\} $ from
$\mathbb{R}^{d}$ to $\mathbb{R}^{d}$ such that all the $S_{e}$
are contracting similarities. The \textit{attractor} of the GD-IFS
is the unique $q$-tuple of nonempty compact sets $\left(K_{1},\ldots,K_{q}\right)$
such that
\begin{equation}
K_{i}=\bigcup_{j=1}^{q}\bigcup_{e\in\mathcal{E}_{i,j}}S_{e}(K_{j}).\label{eq:GDAdef}
\end{equation}
The attractor of a GD-IFS is called a \textit{graph directed attractor}.

Let $\left\{ S_{e}:e\in\mathcal{E}\right\} $ be a GD-IFS. Then every
$S_{e}$ can be uniquely decomposed as
\begin{equation}
S_{e}(x)=r_{e}T_{e}(x)+v_{e}\label{eq:S_edeconposition for GDA}
\end{equation}
for all $x\in\mathbb{R}^{d}$, where $0<r_{e}<1$, $T_{e}$ is an
orthogonal transformation and $v_{e}\in\mathbb{R}^{d}$ is a translation,
for all edges $e$. Let $A^{(s)}$ be the $q\times q$ matrix with
$(i,j)$th entry given by
\begin{equation}
A_{i,j}^{(s)}=\sum_{e\in\mathcal{E}_{i,j}}r_{e}^{s}.\label{eq:A^s def}
\end{equation}
For a matrix $A$ let $\rho(A)$ denote the spectral radius of $A$,
that is the largest absolute value of the eigenvalues of $A$. The
unique solution $s$ of the equation
\begin{equation}
\rho(A^{(s)})=1\label{eq:RO(A^s )=00003D1}
\end{equation}
is called the \textit{similarity dimension} of the GD-IFS.

The directed graph $G\left(\mathcal{V},\mathcal{E}\right)$ is called
\textsl{strongly connected} if for every pair of vertices $i$ and
$j$ there exist a directed path from $i$ to $j$ and a directed
path from $j$ to $i$. We say that the GD-IFS $\left\{ S_{e}:e\in\mathcal{E}\right\} $
is \textsl{strongly connected} if $G\left(\mathcal{V},\mathcal{E}\right)$
is strongly connected. For an overview of the theory of graph directed
attractors see, for example, \cite{Falconer Techniques,Mauldin Williams construction,Wang GDA OSC}.

\subsection{Statement of results}

It is well-known that if $K$ is an attractor of an SS-IFS such that
$\left|\mathcal{T}\right|=1$, where $\left|.\right|$ denotes the
cardinality of a set, then $\Pi_{M}(K)$ is also a self-similar set
for every $l$-dimensional subspace $M$. It was shown by Fraser \cite[Lemma 2.7]{Fraser-box like self affine}
that the vertical and horizontal projections of certain `box-like'
planar self-affine sets are graph directed attractors. We show that,
in the case of finite $\mathcal{T}$, similar results can be obtained
on the structure of the linear images of self-similar sets.

\begin{thm}
\label{cor: SS Lin im is GDA}Let $\left\{ S_{i}\right\} _{i=1}^{m}$
be an SS-IFS with attractor $K\subseteq\mathbb{R}^{d}$ of similarity
dimension $s$ and $L:\mathbb{R}^{d}\longrightarrow\mathbb{R}^{d_{2}}$
be a linear map. Assume that $\mathcal{T}=\left\{ O_{1},\ldots,O_{q}\right\} $
is a finite group where $q=\left|\mathcal{T}\right|$. Then there
exists a strongly connected GD-IFS in $\mathbb{R}^{d_{2}}$ with attractor
$\left(L\circ O_{1}(K),\ldots,L\circ O_{q}(K)\right)$ such that $s$
is the similarity dimension of this GD-IFS with $T_{e}$ the identity
map for all directed edges $e$, and additionally $\mathcal{H}^{s}\left(L\circ O_{1}(K)\right)=\ldots=\mathcal{H}^{s}\left(L\circ O_{q}(K)\right)$.
\end{thm}

Our next result states that if the Hausdorff dimension of $K$ equals
its similarity dimension and $\mathcal{T}$ is finite then we can
always find a projection such that the dimension drops under the projection.
We show this by finding a projection where exact overlapping occurs.
We note that the assumption, that the Hausdorff and the similarity
dimensions are the same, is weaker than the OSC, see \cite[Theorem 1.1]{Pos pack paper}.

\begin{thm}
\label{thm:dimension drop SS}Let $\left\{ S_{i}\right\} _{i=1}^{m}$
be an SS-IFS with attractor $K\subseteq\mathbb{R}^{d}$ of similarity
dimension $s$. Assume that $\mathcal{T}$ is finite and let $l\in\mathbb{N}$,
$l<d$. Then there exists an $l$-dimensional subspace $M\subseteq\mathbb{R}^{d}$
such that $\dim_{H}\left(\Pi_{M}(K)\right)<s$.
\end{thm}

We denote the set of all orthogonal transformations of $\mathbb{R}^{d}$
by $\mathbb{O}_{d}$ which can be metricized using the Euclidean operator
norm
\[
\left\Vert T\right\Vert =\sup_{x\in\mathbb{R}^{d},\left\Vert x\right\Vert =1}\left\Vert Tx\right\Vert ,
\]
where $\left\Vert y\right\Vert $ denotes the Euclidean norm of $y\in\mathbb{R}^{d}$.
With this metric $\mathbb{O}_{d}$ is a compact topological group.
We denote by $\overline{\mathcal{T}}$ the closure of $\mathcal{T}$
in this topology.

The result of Theorem \ref{cor: SS Lin im is GDA}, that $\mathcal{H}^{s}\left(L\circ O_{i}(K)\right)=\mathcal{H}^{s}\left(L\circ O_{j}(K)\right)$,
suggests the following theorem.

\begin{thm}
\label{thm: Egynlo lin Im}Let $\left\{ S_{i}\right\} _{i=1}^{m}$
be an SS-IFS with attractor $K\subseteq\mathbb{R}^{d}$, let $t=\dim_{H}(K)$
and $L:\mathbb{R}^{d}\longrightarrow\mathbb{R}^{d_{2}}$ be a linear
map. If $\mathcal{H}^{t}(K)>0$ then
\begin{equation}
\mathcal{H}^{t}\left(L\circ O(A)\right)=\frac{\mathcal{H}^{t}(A)}{\mathcal{H}^{t}(K)}\mathcal{H}^{t}\left(L(K)\right)\label{eq:ORBIT_EQ}
\end{equation}
for all $A\subseteq K$ and $O\in\overline{\mathcal{T}}$.
\end{thm}

We note that the assumption in Theorem \ref{thm: Egynlo lin Im},
that $\mathcal{H}^{t}(K)>0$, is again a weaker condition than the
OSC (see Example \ref{Ex: nonOSC-line} and Example \ref{Ex: nonOSC-Plane})
and the only role of this assumption is that we can divide by $\mathcal{H}^{t}(K)$
in the formula. If $\mathcal{H}^{t}(K)=0$ then $\mathcal{H}^{t}\left(L(K)\right)=0$
for every linear map $L$. In Example \ref{Ex: nonOSC-Plane} we construct
a self-similar set $K$ with $0<\mathcal{H}^{t}(K)<\infty$ such that
there exists no SS-IFS with attractor $K$ that satisfies the OSC.

Theorem \ref{thm: Egynlo lin Im} has an interesting consequence,
that the linear images of disjoint parts of $K$ are `almost disjoint'
even if no separation condition is satisfied.

\begin{cor}
\label{cor: SS almost disjoint A B}Let $\left\{ S_{i}\right\} _{i=1}^{m}$
be an SS-IFS with attractor $K\subseteq\mathbb{R}^{d}$, let $t=\dim_{H}(K)$,
let $L:\mathbb{R}^{d}\longrightarrow\mathbb{R}^{d_{2}}$ be a linear
map and $A,B\subseteq K$ be such that $\mathcal{H}^{t}\left(A\cap B\right)=0$
and $A$ is $\mathcal{H}^{t}$-measurable. Then $\mathcal{H}^{t}\left(L(A)\cap L(B)\right)=0$.
\end{cor}

In \cite{Eroglu-dense set of rotations} Ero$\breve{\mathrm{g}}$lu
showed that if the transformation group of an SS-IFS in $\mathbb{R}^{2}$
contains a dense set of rotations around the origin then $\mathcal{H}^{s}\left(\Pi_{M}(K)\right)=0$
for all lines $M$, where $s$ denotes the similarity dimension of
the SS-IFS. Ero$\breve{\mathrm{g}}$lu's result does not give any
information about the projections when the OSC is not satisfied. Using
a different approach we generalize this result to higher dimensions
for differentiable maps in place of projections and without any separation
condition, with $s$ replaced by the Hausdorff dimension of $K$.
Both Ero$\breve{\mathrm{g}}$lu's and our proof based on finding two
or several cylinders and fixed direction in which their projection
have large overlap. Ero$\breve{\mathrm{g}}$lu uses similar arguments
to those of Simon and Solomyak \cite{Simon-Solomyak-Visibility} to
show that the projection measure has infinite upper density almost
everywhere, whilst we use a Vitali covering argument to show that
the Hausdorff measure of the image must collapse.

Let $0<l\leq d$ be integers and let $G_{d,l}$ denote the \textit{Grassmann
manifold} of $l$-dimensional linear subspaces of $\mathbb{R}^{d}$
equipped with the usual topology (see for example \cite[Section 3.9]{Mattilakonyv}).

\begin{thm}
\label{cor:Dense group lin image 0}Let $\left\{ S_{i}\right\} _{i=1}^{m}$
be an SS-IFS with attractor $K$, let $t=\dim_{H}(K)$, let $U$ be
an open neighbourhood of $K$ and assume that there exists $M\in G_{d,l}$
such that the set $\left\{ O(M):O\in\mathcal{T}\right\} $ is dense
in $G_{d,l}$ for some $1\leq l<d$. Then $\mathcal{H}^{t}\left(g(K)\right)=0$
for every continuously differentiable map $g:U\longrightarrow\mathbb{R}^{d_{2}}$
such that $\mathrm{rank}(g'(x))\leq l$ for every $x\in K$.
\end{thm}

If $\mathrm{rank}(g'(x))=d$ for some $x\in K$ then $g$ is a bi-Lipschitz
function between a neigbourhood $V$ of $x$ and $g(V)$ and hence
$\mathcal{H}^{t}\left(g(K)\right)=0$ if and only if $\mathcal{H}^{t}(K)=0$.

We note that the assumption, that $\left\{ O(M):O\in\mathcal{T}\right\} $
is dense for some $M\in G_{d,l}$, is equivalent to that $\left\{ O(M):O\in\mathcal{T}\right\} $
is dense for each $M\in G_{d,l}$.

It was shown by Peres and Shmerkin \cite[Theorem 5]{Peres-Schmerkin Resonance between Cantor sets}
on the plane under the conditions of Theorem \ref{cor:Dense group lin image 0}
that $\dim_{H}\left(\Pi_{M}(K)\right)=\min\left\{ t,1\right\} $ for
every line $M$. This was generalized to higher dimensions by Hochman
and Shmerkin \cite[Corollary 1.7]{Hochman-Schmerkin-local entropy}
for SS-IFS that satisfies the SSC and the SSC was relaxed by Falconer
and Jin \cite[Corollary 5.2]{Falconer-Jin} to the `strong variational
principle'. We use the result of Hochman and Shmerkin and a dimension
approximation method (Proposition \ref{lem:dim approx lem}) to deduce
the same conclusion without any separation condition.

\begin{thm}
\label{thm:Dense group dim conserv}Let $\left\{ S_{i}\right\} _{i=1}^{m}$
be an SS-IFS with attractor $K$, let $t=\dim_{H}(K)$, let $U$ be
an open neighbourhood of $K$ and assume that there exists $M\in G_{d,l}$
such that the set $\left\{ O(M):O\in\mathcal{T}\right\} $ is dense
in $G_{d,l}$ for some $1\leq l<d$. Then $\dim_{H}\left(g(K)\right)=\min\left\{ t,l\right\} $
for every continuously differentiable map $g:U\longrightarrow\mathbb{R}^{l}$
such that $\mathrm{rank}(g'(x))=l$ for some $x\in K$.
\end{thm}

We can state a corollary of Theorem \ref{thm:Dense group dim conserv}
which applies to $g:U\longrightarrow\mathbb{R}^{d_{2}}$ where $d_{2}$
may be greater than $l$.
\begin{cor}
\label{cor:Dense diff dim conserv COR}Let $\left\{ S_{i}\right\} _{i=1}^{m}$
be an SS-IFS with attractor $K$, let $t=\dim_{H}(K)$, let $U$ be
an open neighbourhood of $K$ and assume that there exists $M\in G_{d,l}$
such that the set $\left\{ O(M):O\in\mathcal{T}\right\} $ is dense
in $G_{d,l}$ for some $1\leq l<d$. If $g:U\longrightarrow\mathbb{R}^{d_{2}}$
is a continuously differentiable map such that $\mathrm{rank}(g'(x))=l$
for every $x\in K$ and either of the following conditions is satisfied

(i) $g\in C^{\infty}$,

(ii) $t\leq l$,

\noindent then $\dim_{H}\left(g(K)\right)=\min\left\{ t,l\right\} $.
\end{cor}

In the planar case $\left|\mathcal{T}\right|=\infty$ is equivalent
to that $\left\{ O(M):O\in\mathcal{T}\right\} $ is dense in $G_{2,1}$
for every $M\in G_{2,1}$. Furthermore, it can be easily shown that
in the planar case $\left|\mathcal{T}\right|=\infty$ also implies
that $\mathcal{T}$ contains a rotation of infinite order. Example
\ref{Ex: vegyes proj} shows that in general $\left|\mathcal{T}\right|=\infty$
does not imply either the conclusion of Theorem \ref{thm:Dense group dim conserv}
or the conclusion of Theorem \ref{cor:Dense group lin image 0} in
higher dimensions.

Example \ref{cannot replace} shows that neither the conclusion of
Theorem \ref{thm:Dense group dim conserv} nor the conclusion of Theorem
\ref{cor:Dense group lin image 0} necessarily remain true if we replace
$g$ with a Lipschitz function that is a composition of an orthogonal
projection and a bi-Lipschitz map.

In \cite{Furstenberg-dimension conservation} Furstenberg introduces
the definition of a `dimension conserving map'. If $f:A\longrightarrow\mathbb{R}^{d_{2}}$
is a Lipschitz map where $A\subseteq\mathbb{R}^{d}$ we say that $f$
is \textsl{dimension conserving} if, for some $\delta\geq0$,
\[
\delta+\dim_{H}\left\{ y\in f(A):\dim_{H}(f^{-1}(y))\geq\delta\right\} \geq\dim_{H}A
\]
with that convention that $\dim_{H}(\emptyset)=-\infty$ so that $\delta$
cannot be chosen too large. Furstenberg also introduces `mini- and
micro-sets of a set', and a compact set is defined to be `homogeneous'
if all of its micro-sets are also mini-sets. Furstenberg`s main theorem
\cite[Theorem 6.2]{Furstenberg-dimension conservation} states that
the restriction of a linear map to a homogeneous compact set is dimension
conserving. He suggests that if $K$ is a self-similar set, $\mathcal{T}$
has only one element and the SSC is satisfied then $K$ is homogeneous.
One can show that $K$ is homogeneous even if $\mathcal{T}$ is finite
and the SSC is satisfied. Thus for such $K$ the restriction of any
linear map to $K$ is dimension conserving even though, by Theorem
\ref{thm:dimension drop SS}, there must be a projection under which
the dimension drops. Theorem \ref{thm:Dense group dim conserv} implies
that if $\left\{ O(M):O\in\mathcal{T}\right\} $ is dense in $G_{d,l}$
where $\dim_{H}K\leq l$, then the restriction of $g$ to $K$ is
dimension conserving, where $g$ is a continuously differentiable
map of rank $l$.

The following proposition is a useful tool for generalizing results
about Hausdorff dimension in case of SSC to the case with no separation
condition.

\begin{prop}
\label{lem:dim approx lem}Let $\left\{ S_{i}\right\} _{i=1}^{m}$
be an SS-IFS with attractor $K$. For all $\varepsilon>0$ there exists
an SS-IFS $\left\{ \widehat{S_{i}}\right\} _{i=1}^{\widehat{m}}$
that satisfies the SSC with attractor $\widehat{K}$ such that $\widehat{K}\subseteq K$,
$\dim_{H}K-\varepsilon<\dim_{H}\widehat{K}$ and the transformation
group $\widehat{\mathcal{T}}$ of $\left\{ \widehat{S_{i}}\right\} _{i=1}^{\widehat{m}}$
is dense in $\mathcal{T}$.
\end{prop}

The planar case of Proposition \ref{lem:dim approx lem} was known
before and was used, for example, in \cite{Peres-Schmerkin Resonance between Cantor sets,Orponen-distance set conjecture}.
The proof in the planar case is not difficult and in three dimensions
is not more complicated. However, the higher dimensional case is more
subtle and the proof relies on Kronecker's simultaneous approximation
theorem \cite[Theorem 442]{Hardy-Wright-An introduction to the theory of numbers}.

All our results are valid without assuming OSC. The following proposition
develops a new tool that serves the role of separation conditions
in the proofs. In Section \ref{sec:Main-Lemma} we state two other
variants of Proposition \ref{lem:Egyenlito disjoint lem T=00003Dinfty}
and we hope that such variants of Proposition \ref{lem:Egyenlito disjoint lem T=00003Dinfty}
may help to extend other results to settings without any separation
condition. The proof of Proposition \ref{lem:Egyenlito disjoint lem T=00003Dinfty}
relies on Vitali's covering theorem \cite[Theorem 1.10]{falconer geom of fract sets}.

\begin{prop}
\label{lem:Egyenlito disjoint lem T=00003Dinfty}Let $\left\{ S_{i}\right\} _{i=1}^{m}$
be an SS-IFS with attractor $K$ and let $t=\dim_{H}(K)$. Let $O\in\overline{\mathcal{T}}$
be arbitrary and $\delta>0$. Then there exists $\mathcal{I}_{\infty}\subseteq\bigcup_{k=1}^{\infty}\mathcal{I}^{k}$
such that $\left\Vert T_{\boldsymbol{\mathbf{i}}}-O\right\Vert <\delta$
for all $\boldsymbol{\mathbf{i}}\in\mathcal{I}_{\infty}$, $K_{\boldsymbol{\mathbf{i}}}\cap K_{\boldsymbol{\mathbf{j}}}=\emptyset$
for $\boldsymbol{\mathbf{i}},\boldsymbol{\mathbf{j}}\in\mathcal{I}_{\infty}$,
$\boldsymbol{\mathbf{i}}\neq\boldsymbol{\mathbf{j}}$, and $\mathcal{H}^{t}\left(K\setminus\left(\bigcup_{\boldsymbol{\mathbf{i}}\in\mathcal{I}_{\infty}}K_{\boldsymbol{\mathbf{i}}}\right)\right)=0$.
\end{prop}

\begin{rem}
\label{Rem: Mainlem EQ}Under the assumptions of Proposition \ref{lem:Egyenlito disjoint lem T=00003Dinfty}
if $K$ is a $t$-set it follows that $\mathcal{H}^{t}(K)=\sum_{\boldsymbol{\mathbf{i}}\in\mathcal{I}_{\infty}}\mathcal{H}^{t}(K_{\boldsymbol{\mathbf{i}}})=\sum_{\boldsymbol{\mathbf{i}}\in\mathcal{I}_{\infty}}r_{\boldsymbol{\mathbf{i}}}^{t}\cdot\mathcal{H}^{t}(K)$,
hence
\[
\sum_{\boldsymbol{\mathbf{i}}\in\mathcal{I}_{\infty}}r_{\boldsymbol{\mathbf{i}}}^{t}=1.
\]
This equation plays the role of (\ref{eq: S-dim sum}) in the non-OSC
case.
\end{rem}
Another advantage of Proposition \ref{lem:Egyenlito disjoint lem T=00003Dinfty}
is that we can regard the IFS as one for which the orthogonal part
$T_{\boldsymbol{\mathbf{i}}}$ of the maps are approximately the same
at any level. This observation helps us to deal with the higher dimensional
cases when the rotations do not necessarily commute.

The Hausdorff content of a set $H\subseteq\mathbb{R}^{d}$ is $\mathcal{H}_{\infty}^{t}(H)=\inf\left\{ \sum_{i=1}^{\infty}\mathrm{diam}(U_{i})^{t}:H\subseteq\bigcup_{i=1}^{\infty}U_{i}\right\} $.
The next proposition says that the Hausdorff measure and content of
linear images of $K$ coincide. It follows that the Hausdorff measure
is upper semi-continuous since the Hausdorff content is upper semi-continuous.
This observation is essential in the proofs of Theorem \ref{thm: Egynlo lin Im}
and Theorem \ref{cor:Dense group lin image 0}.

\begin{prop}
\label{lem:Lin neighbourhood lem}Let $\left\{ S_{i}\right\} _{i=1}^{m}$
be an SS-IFS with attractor $K\subseteq\mathbb{R}^{d}$, let $t=\dim_{H}(K)$
and $L:\mathbb{R}^{d}\longrightarrow\mathbb{R}^{d_{2}}$ be a linear
map. Then

(i) \textup{$\mathcal{H}^{t}(L(K))=\mathcal{H}_{\infty}^{t}(L(K))$,}

(ii) for every $\varepsilon>0$ there exists $\delta>0$ such that
for every linear map $L_{2}:\mathbb{R}^{d}\longrightarrow\mathbb{R}^{d_{2}}$
with $\left\Vert L-L_{2}\right\Vert <\delta$ we have that $\mathcal{H}^{t}(L_{2}(K))\leq\mathcal{H}^{t}(L(K))+\varepsilon$.
\end{prop}

In particular Proposition \ref{lem:Lin neighbourhood lem} implies
that $\mathcal{H}^{t}(\Pi_{M}(K))$ is upper semi-continuous in $M\in G_{d,l}$,
in contrast with a result of Hochman and Shmerkin \cite[Theorem 1.8]{Hochman-Schmerkin-local entropy}
on the lower semi-continuity of the lower Hausdorff dimension of Bernoulli
convolutions.

It is well-know that if for a set $H$ with $\mathcal{H}^{t}(H)<\infty$
we have that $\mathcal{H}^{t}(H)=\mathcal{H}_{\infty}^{t}(H)$ then
for every $\mathcal{H}^{t}$-measurable subset $A\subseteq H$
\[
\mathcal{H}^{t}(H)=\mathcal{H}^{t}(A)+\mathcal{H}^{t}(H\setminus A)\geq\mathcal{H}_{\infty}^{t}(A)+\mathcal{H}_{\infty}^{t}(H\setminus A)\geq\mathcal{H}_{\infty}^{t}(H)=\mathcal{H}^{t}(H)
\]
so $\mathcal{H}^{t}(A)=\mathcal{H}_{\infty}^{t}(A)$.

It follows that $\mathcal{H}^{t}(B)=\mathcal{H}_{\infty}^{t}(B)$
for every $\mathcal{H}^{t}$-measurable subset $B\subseteq L(K)$.
Let $A\subseteq L(K)$ be arbitrary and $B\subseteq L(K)$ be a $\mathcal{H}^{t}$-measurable
hull of $A$ (for the definition of $\mathcal{H}^{t}$-measurable
hull see Section \ref{sub:Linear-images}). We can further assume
that $\mathrm{diam}(B)=\mathrm{diam}(A)$. Then $\mathcal{H}^{t}(A)=\mathcal{H}^{t}(B)=\mathcal{H}_{\infty}^{t}(B)\leq\mathrm{diam}(A)^{t}$.
In particular for $L=Id_{\mathbb{R}^{d}}$ we obtain that $\mathcal{H}^{t}(A)\leq\mathrm{diam}(A)^{t}$
for every subset $A\subseteq K$. Thus Proposition \ref{prop:.diam^s}
remains valid with $s$ replaced by $t$.

Proposition \ref{lem:Lin neighbourhood lem} does not generalise to
smooth maps. If $K$ is a $1$-set and $L:\mathbb{R}^{d}\longrightarrow\mathbb{R}$
is such that $L(K)=[0,1]$ (see Example \ref{Ex: poz proj sierp}
for $t=1$) then $g(x):=\left(\cos(L(x),\sin(L(x)\right):\mathbb{R}^{d}\longrightarrow\mathbb{R}^{2}$
is such that $\mathcal{H}^{t}(g(K))\neq\mathcal{H}_{\infty}^{t}(g(K))$.

\section{Preliminaries\label{sec:Preliminaries}}

In this section we summarize the background and preliminary results
needed to prove our main results.

Let $H\subset\mathbb{R}^{d}$ and write $\dim_{H}(H)$, $\underline{\dim}_{B}(H)$,
$\overline{\dim}_{B}(H)$ and $\dim_{P}(H)$ for the Hausdorff dimension,
lower box dimension, upper box dimension and packing dimension of
$H$ respectively (for the definitions see for example \cite{Falconer1}).
We denote the diameter of $H$ by $\mathrm{diam}(H)$. We recall that
$\mathcal{H}^{s}(H)$ denotes the $s$-dimensional Hausdorff measure
of $H$ and $\left|A\right|$ denotes the cardinality of a set $A$.
If $x\in\mathbb{R}^{d}$ and $r>0$ then we denote the open ball that
is centered at $x$ with radius $r$ by $B(x,r)$. We denote the identity
map on $\mathbb{R}^{d}$ by $Id_{\mathbb{R}^{d}}$.

\subsection{Orthogonal transformations}

Let $L:\mathbb{R}^{d}\longrightarrow\mathbb{R}^{d_{2}}$ be a linear
map. We recall that the Euclidean operator norm of $L$ is defined
as
\[
\left\Vert L\right\Vert =\sup_{x\in\mathbb{R}^{d},\left\Vert x\right\Vert =1}\left\Vert Lx\right\Vert 
\]
where $\left\Vert y\right\Vert $ denotes the Euclidean norm of a
vector $y$. A linear map $T:\mathbb{R}^{d}\longrightarrow\mathbb{R}^{d}$
is called \textit{orthogonal} if it preserves the norm, i.e. $\left\Vert T(y)\right\Vert =\left\Vert y\right\Vert $
for all $y\in\mathbb{R}^{d}$, hence $\left\Vert T\right\Vert =1$.
If $L$ is linear and $T$ is orthogonal then it follows that $\left\Vert L\right\Vert =\left\Vert L\circ T\right\Vert $.
Similarly if $T:\mathbb{R}^{d_{2}}\longrightarrow\mathbb{R}^{d_{2}}$
is an orthogonal transformation and $L$ is linear as above then $\left\Vert L\right\Vert =\left\Vert T\circ L\right\Vert $.
We recall that $\mathbb{O}_{d}$ denotes the set of all orthogonal
transformations. With the Euclidean operator norm metric $\mathbb{O}_{d}$
is a compact topological group.

\begin{lem}
\label{lem:lem_i2}If $T_{1},\ldots,T_{m}\in\mathbb{O}_{d}$ then
the semigroup generated by $T_{1},\ldots,T_{m}$ is dense in the group
generated by $T_{1},\ldots,T_{m}$.
\end{lem}

Lemma \ref{lem:lem_i2} is a consequence of the compactness of $\mathbb{O}_{d}$.

\begin{prop}
\label{lem:Jordanos lem}If $T\in\mathbb{O}_{d}$ then for all $N\in\mathbb{N}$
there exists $k\in\mathbb{N}$, $k\geq N$, such that the group generated
by $T^{k}$ is dense in the group generated by $T$.
\end{prop}

\noindent Proposition \ref{lem:Jordanos lem} can be deduced from
the fact that every compact abelian Lie group is isomorphic to the
product of finitely many circle groups and a finite abelian group.
However, we provide an elementary proof relying on Kronecker's simultaneous
approximation theorem.

\begin{proof}
By \cite[Theorem 10.12]{Blyth-Robertson Further Linear Algebra} we
can find an orthonormal basis in $\mathbb{R}^{d}$ with respect to
which the matrix form of $T$ is block diagonal such that the blocks
are either $\begin{bmatrix}1\end{bmatrix}$ or $\begin{bmatrix}-1\end{bmatrix}$
or $B(\alpha_{i})=\begin{bmatrix}\cos(\alpha_{i}) & -\sin(\alpha_{i})\\
\sin(\alpha_{i}) & \cos(\alpha_{i})
\end{bmatrix}$ for some $\alpha_{1},\ldots,\alpha_{n}\in[0,2\pi)$. Let $\mathcal{J}=\left\{ i\in\left\{ 1,\ldots,n\right\} :B(\alpha_{i})\:\mathrm{has\:finite\:order}\right\} $.
If $i\in\mathcal{J}$ then let $k_{i}$ be the order of $B(\alpha_{i})$.
Let $k_{0}=2^{N}\cdot\prod_{i\in\mathcal{J}}k_{i}>N$. Then for any
$l\in\mathbb{N}$ it follows that $B(\alpha_{i})^{l\cdot k_{0}+1}=B(\alpha_{i})$
for all $i\in\mathcal{J}$, $\begin{bmatrix}1\end{bmatrix}^{l\cdot k_{0}+1}=\begin{bmatrix}1\end{bmatrix}$
and $\begin{bmatrix}-1\end{bmatrix}^{l\cdot k_{0}+1}=\begin{bmatrix}-1\end{bmatrix}$.

Let $\mathcal{A}=\left\{ \frac{\alpha_{i}}{2\pi}:i\in\left\{ 1,\ldots,n\right\} \setminus\mathcal{J}\right\} $.
If $\frac{\alpha_{i}}{2\pi}\in\mathcal{A}$ then $\frac{\alpha_{i}}{2\pi}$
is irrational. Let $1,\frac{\beta_{1}}{2\pi},\ldots,\frac{\beta_{m}}{2\pi}\in\mathcal{A}\bigcup\left\{ 1\right\} $
be a maximal linearly independent system over $\mathbb{Q}$. Then
we can write every $\alpha_{i}$ for $i\in\left\{ 1,\ldots,n\right\} \setminus\mathcal{J}$
in the form $\alpha_{i}=\left(\frac{p_{i,0}}{q_{i,0}}\cdot2\pi+\sum_{j=1}^{m}\frac{p_{i,j}}{q_{i,j}}\cdot\beta_{j}\right)$
such that $p_{i,j}\in\mathbb{Z}$, $q_{i,j}\in\mathbb{N}$. Let $M=\max\left\{ \left|\frac{p_{i,j}}{q_{i,j}}\right|:i\in\left\{ 1,\ldots,n\right\} \setminus\mathcal{J},j\in\left\{ 0,\ldots,m\right\} \right\} $,
let $\widehat{q}=\prod_{i\in\left\{ 1,\ldots,n\right\} \setminus\mathcal{J}}\prod_{j=1}^{m}q_{i,j}$,
let $q=\prod_{i\in\left\{ 1,\ldots,n\right\} \setminus\mathcal{J}}q_{i,0}$
and let $k=k_{0}\cdot q+1$. Let $\delta>0$ be arbitrary. Then $1,(k-1)\cdot k\cdot\frac{\beta_{1}}{2\pi},\ldots,(k-1)\cdot k\cdot\frac{\beta_{m}}{2\pi}$
is a linearly independent system over $\mathbb{Q}$, hence by Kronecker's
simultaneous approximation theorem \cite[Theorem 442]{Hardy-Wright-An introduction to the theory of numbers}
we can find $p\in\mathbb{N}$, $d_{1},\ldots,d_{m}\in\mathbb{Z}$
such that $\left|p\cdot(k-1)\cdot k\cdot\frac{\beta_{j}}{2\pi}-(1-k)\cdot\frac{\beta_{j}}{2\pi}-d_{j}\right|<\frac{\delta}{M\cdot m\cdot2\pi}$
for $j=1,\ldots,m$ and we can further assume that $p,d_{1},\ldots,d_{m}$
are multiples of $\widehat{q}$. It follows that $\left|\left(p\cdot(k-1)+1\right)\cdot k\cdot\beta_{j}-\beta_{j}-d_{j}\cdot2\pi\right|<\frac{\delta}{M\cdot m}$.
By the choice of $q$ and $k$ the numbers defined by $D_{i}=\left(p\cdot(k-1)+1\right)\cdot k\cdot\frac{p_{i,0}}{q_{i,0}}-\frac{p_{i,0}}{q_{i,0}}$
are integers for all $i\in\left\{ 1,\ldots,n\right\} \setminus\mathcal{J}$.
Thus
\[
\left|\left(p\cdot\left(k-1\right)+1\right)\cdot k\cdot\alpha_{i}-\alpha_{i}-D_{i}\cdot2\pi-\sum_{j=1}^{m}\frac{p_{i,j}}{q_{i,j}}\cdot d_{j}\cdot2\pi\right|
\]
\[
=\left|\sum_{j=1}^{m}\frac{p_{i,j}}{q_{i,j}}\left(\left(p\cdot\left(k-1\right)+1\right)\cdot k\cdot\beta_{j}-\beta_{j}\right)-\frac{p_{i,j}}{q_{i,j}}\cdot d_{j}\cdot2\pi\right|
\]
\[
\leq\sum_{j=1}^{m}\left|\frac{p_{i,j}}{q_{i,j}}\right|\cdot\left|\left(p\cdot\left(k-1\right)+1\right)\cdot k\cdot\beta_{j}-\beta_{j}-d_{j}\cdot2\pi\right|<\delta
\]
for $i\in\left\{ 1,\ldots,n\right\} \setminus\mathcal{J}$ and by
the choice of $\widehat{q}$ we have that $\sum_{j=1}^{m}\frac{p_{i,j}}{q_{i,j}}\cdot d_{j}\in\mathbb{Z}$.
So if we set $z=\left(p\cdot(k-1)+1\right)$ then $\begin{bmatrix}1\end{bmatrix}^{k\cdot z}=\begin{bmatrix}1\end{bmatrix}$,
$\begin{bmatrix}-1\end{bmatrix}^{k\cdot z}=\begin{bmatrix}-1\end{bmatrix}$,
$B(\alpha_{i})^{k\cdot z}=B(\alpha_{i})$ for all $i\in\mathcal{J}$
and $B(\alpha_{i})^{k\cdot z}=B(\gamma_{i})$ for all $i\in\left\{ 1,\ldots,n\right\} \setminus\mathcal{J}$
for some $\gamma_{i}\in(0,2\pi)$ such that $\left|\gamma_{i}-\alpha_{i}\right|<\delta$.

So we can approximate $T$ by the powers of $T^{k}$, hence we can
approximate the powers of $T$ by the powers of $T^{k}$. Thus the
group generated by $T^{k}$ is dense in the group generated by $T$.
\end{proof}

\subsection{\label{sub:Linear-images}Linear images}

Every linear map $L:\mathbb{R}^{d}\longrightarrow\mathbb{R}^{d_{2}}$
is a Lipschitz map with Lipschitz constant $\left\Vert L\right\Vert $.

Giving the following well-known lemma, which we use without reference
throughout the paper.
\begin{lem}
\label{lem:H^s(vet)<H^s}Let $H\subseteq\mathbb{R}^{d}$ and $L:\mathbb{R}^{d}\longrightarrow\mathbb{R}^{d_{2}}$
be a linear map. Then $\mathcal{H}^{s}\left(L(H)\right)\leq\left\Vert L\right\Vert ^{s}\cdot\mathcal{H}^{s}\left(H\right)$
and $\dim_{H}(L(H))\leq\dim_{H}(H)$.
\end{lem}

A set $A\subseteq\mathbb{R}^{d}$ is called \textit{$\mathcal{H}^{s}$-measurable}
if $\mathcal{H}^{s}(H)=\mathcal{H}^{s}(H\setminus A)+\mathcal{H}^{s}(H\cap A)$
for any set $H\subseteq\mathbb{R}^{d}$. For an arbitrary set $A$
a $\mathcal{H}^{s}$-measurable set $B$ such that $A\subseteq B$
and $\mathcal{H}^{s}(A)=\mathcal{H}^{s}(B)$ is called a \textit{$\mathcal{H}^{s}$-measurable
hull of $A$}.

\begin{lem}
\label{lem: hull exist}For every set $A\subseteq\mathbb{R}^{d}$
there exists a Borel set $B$ that is a $\mathcal{H}^{s}$-measurable
hull of $A$.
\end{lem}

From \cite[Theorem 4.4]{Mattilakonyv} it can be deduced that there
exists a $\mathcal{H}^{s}$-measurable hull of $A$, that is a $G_{\delta}$
set. So there exists a Borel set $B$ that is a $\mathcal{H}^{s}$-measurable
hull of $A$.

\begin{lem}
\label{lem:Hull Lemma}For a set $A\subseteq\mathbb{R}^{d}$ such
that $\mathcal{H}^{s}(A)<\infty$, let $B$ be a $\mathcal{H}^{s}$-measurable
hull of $A$ and $L:\mathbb{R}^{d}\longrightarrow\mathbb{R}^{d_{2}}$
be a linear map. Then $\mathcal{H}^{s}\left(L(A)\right)=\mathcal{H}^{s}\left(L(B)\right)$.\end{lem}
\begin{proof}
Let $C$ be a $\mathcal{H}^{s}$-measurable hull of $L(A)$ such that
$C\subseteq L(B)$. Then $B\cap L^{-1}(C)$ is also a $\mathcal{H}^{s}$-measurable
hull of $A$. It follows that $\mathcal{H}^{s}(B\cap L^{-1}(C))=\mathcal{H}^{s}(A)=\mathcal{H}^{s}(B)$,
thus $\mathcal{H}^{s}\left(B\setminus B\cap L^{-1}(C)\right)=0$.
Hence $\mathcal{H}^{s}\left(L(B)\setminus C\right)=0$ and so $\mathcal{H}^{s}\left(L(A)\right)=\mathcal{H}^{s}\left(L(B)\right)$.
\end{proof}

\begin{lem}
\label{lem: Mag lemma}Let $0<l<d$ be integers, $v\in\mathbb{R}^{d}$,
$L:\mathbb{R}^{d}\longrightarrow\mathbb{R}^{d_{2}}$ be a linear map
with $\mathrm{\mathrm{rank}}(L)=l$ and $\mathcal{T}\subseteq\mathbb{O}_{d}$
be such that there exists $M\in G_{d,l}$ such that the set $\left\{ O(M):O\in\mathcal{T}\right\} $
is dense in $G_{d,l}$. Then there exists $O_{0}\in\overline{\mathcal{T}}$
such that $L\circ O_{0}(v)=0$.
\end{lem}

\begin{proof}
It is easy to see that if there exists $M\in G_{d,l}$ such that the
set $\left\{ O(M):O\in\mathcal{T}\right\} $ is dense in $G_{d,l}$
then $\left\{ O(M):O\in\mathcal{T}\right\} $ is dense in $G_{d,l}$
for every $M\in G_{d,l}$. Let $M\in G_{d,l}$ be such that $M$ is
contained in the orthogonal complement of $v$. Since $\mathrm{\mathrm{rank}}(L)=l$
it follows that $\dim\left(\mathrm{Ker}(L)\right)=d-l$. Hence $\dim\left(\mathrm{Ker}(L)^{\bot}\right)=l$.
There exists $O_{0}\in\overline{\mathcal{T}}$ such that $O_{0}(M)=\mathrm{Ker}(L)^{\bot}$.
Since $v$ is orthogonal to $M$ it follows that $O_{0}(v)$ is orthogonal
to $O_{0}(M)=\mathrm{Ker}(L)^{\bot}$. Thus $O_{0}(v)\in\mathrm{Ker}(L)$,
so $L\circ O_{0}(v)=0$.
\end{proof}

\subsection{Vitali's covering theorem}

Let $H\subset\mathbb{R}^{d}$. A collection of sets $\mathcal{A}$
is called a Vitali cover of $H$ if for each $x\in H$, $\delta>0$
there exists $A\in\mathcal{A}$ with $x\in A$ and $0<\mathrm{diam}(A)<\delta$.

\begin{prop}
\label{prop:Vtali kov}Let $H\subset\mathbb{R}^{d}$ be a $\mathcal{H}^{s}$-measurable
set with $\mathcal{H}^{s}(H)<\infty$ and $B\subset\mathbb{R}^{d}$
be a closed set with $0<\mathrm{diam}(B)<\infty$ and $0<\mathcal{H}^{s}(B)<\infty$.
Let $\mathcal{A}$ be a Vitali cover of $H$ such that every element
of $\mathcal{A}$ is similar to $B$ and every element of $\mathcal{A}$
is contained in $H$. Then there exists a disjoint sequence of sets
(finite or countable) $A_{1},A_{2},\ldots\in\mathcal{A}$ such that
$\mathcal{H}^{s}\left(H\setminus\left(\bigcup_{i=1}^{\infty}A_{i}\right)\right)=0$.
\end{prop}

Proposition \ref{prop:Vtali kov} follows from a version of Vitalis's
covering theorem \cite[Theorem 1.10]{falconer geom of fract sets}
because $\sum_{i=1}^{\infty}\mathrm{diam}(A_{i})^{s}=\infty$ is not
possible since by the similarity we have that $\sum_{i=1}^{\infty}\mathrm{diam}(A_{i})^{s}=\sum_{i=1}^{\infty}\mathcal{H}^{s}(A_{i})\cdot\frac{\mathrm{diam}(B)^{s}}{\mathcal{H}^{s}(B)}\leq\mathcal{H}^{s}(H)\cdot\frac{\mathrm{diam}(B)^{s}}{\mathcal{H}^{s}(B)}<\infty$.

\subsection{Self-similar sets}
\begin{prop}
\label{prop:.diam^s}Let $\left\{ S_{i}\right\} _{i=1}^{m}$ be an
SS-IFS with attractor $K$ and $s$ be the similarity dimension of
$\left\{ S_{i}\right\} _{i=1}^{m}$. Then $\mathcal{H}^{s}(K)\leq\mathrm{diam}(K)^{s}<\infty$.
\end{prop}

For details see \cite[5.1 Prop(4)]{Hutchinson}.
\begin{prop}
\label{prop:.implicitkov}Let $\left\{ S_{i}\right\} _{i=1}^{m}$
be an SS-IFS with attractor $K$. Then $\dim_{H}(K)=\underline{\dim}_{B}(K)=\overline{\dim}_{B}(K)=\dim_{P}(K)$
and $\mathcal{H}^{t}(K)<\infty$ where $t=\dim_{H}(K)$.
\end{prop}

Proposition \ref{prop:.implicitkov} can be deduced by an application
of implicit methods \cite[Thm 3.2]{Falconer Techniques}.

\begin{prop}
\label{prop: OSC schief prop}Let $\left\{ S_{i}\right\} _{i=1}^{m}$
be an SS-IFS with attractor $K$ and let $s$ be the similarity dimension
of $\left\{ S_{i}\right\} _{i=1}^{m}$. Then the following are equivalent:

(i) $\left\{ S_{i}\right\} _{i=1}^{m}$ satisfies the OSC

(ii) $\dim_{H}(K)=s$ and $0<\mathcal{H}^{s}(K)<\infty$

(iii) $0<\mathcal{H}^{s}(K)$.
\end{prop}

For details see \cite{Schief OSC}.

\subsection{Graph directed attractors}

We say that the GD-IFS $\left\{ S_{e}:e\in\mathcal{E}\right\} $ satisfies
the \textit{strong separation condition} (SSC) if
\[
K_{i}=\bigcup_{j=1}^{q}\bigcup_{e\in\mathcal{E}_{i,j}}S_{e}(K_{j})
\]
is a disjoint union for each $i\in\mathcal{V}$. We say that the GD-IFS
$\left\{ S_{e}:e\in\mathcal{E}\right\} $ satisfies the \textit{open
set condition} (OSC) if there exists a $q$-tuple of nonempty open
sets $\left(U_{1},\ldots,U_{q}\right)$ such that
\[
\bigcup_{j=1}^{q}\bigcup_{e\in\mathcal{E}_{i,j}}S_{e}(U_{j})\subseteq U_{i}
\]
and the union is disjoint for each $i\in\mathcal{V}$. It is easy
to see that SSC implies OSC for GD-IFSs.

\begin{prop}
\label{prop:finite measure GDA}Let $\left\{ S_{e}:e\in\mathcal{E}\right\} $
be a strongly connected GD-IFS with attractor $\left(K_{1},\ldots,K_{q}\right)$
and let $s$ be the similarity dimension of $\left\{ S_{e}:e\in\mathcal{E}\right\} $.
Then $\mathcal{H}^{s}(K_{i})<\infty$ for $i\in\mathcal{V}$.
\end{prop}

For the details of the proof of Proposition \ref{prop:finite measure GDA}
see \cite[p. 172]{Edgar-Measure-Topolog-first-edition}.

\begin{prop}
\label{prop:implicit GDA prop}Let $\left\{ S_{e}:e\in\mathcal{E}\right\} $
be a strongly connected GD-IFS with attractor $\left(K_{1},\ldots,K_{q}\right)$.
Then $\dim_{H}(K_{j})=\dim_{H}(K_{i})=\underline{\dim}_{B}(K_{i})=\overline{\dim}_{B}(K_{i})=\dim_{P}(K_{i})$
for each $i,j\in\mathcal{V}$ and $\mathcal{H}^{t}(K_{i})<\infty$
for each $i\in\mathcal{V}$ where $t=\dim_{H}(K_{i})$.
\end{prop}

Proposition \ref{prop:implicit GDA prop} can be deduced by the application
of the implicit methods \cite[Theorem 3.2]{Falconer Techniques}.

\begin{prop}
\label{prop: OSC GDA Wang}Let $\left\{ S_{e}:e\in\mathcal{E}\right\} $
be a strongly connected GD-IFS with attractor $\left(K_{1},\ldots,K_{q}\right)$
and let $s$ be the similarity dimension of $\left\{ S_{e}:e\in\mathcal{E}\right\} $.
Then the following are equivalent:

(i) $\left\{ S_{e}:e\in\mathcal{E}\right\} $ satisfies the OSC

(ii) $\dim_{H}(K_{i})=s$ and $0<\mathcal{H}^{s}(K_{i})<\infty$ for
each $i\in\mathcal{V}$

(iii) $0<\mathcal{H}^{s}(K_{i})$ for some $i\in\mathcal{V}$.
\end{prop}

For the details of the proof of Proposition \ref{prop: OSC GDA Wang}
see \cite{Wang GDA OSC}.

\subsection{Irreducible matrices}

Recall that for a GD-IFS $\left\{ S_{e}:e\in\mathcal{E}\right\} $
we define $A^{(s)}$ as in (\ref{eq:A^s def}) $A_{i,j}^{(s)}=\sum_{e\in\mathcal{E}_{i,j}}r_{e}^{s}$.
Then for the $k$th power of $A^{(s)}$ it follows that
\[
\left(A^{(s)}\right)_{i,j}^{k}=\sum_{\boldsymbol{\mathbf{e}}\in\mathcal{E}_{i,j}^{k}}r_{\boldsymbol{\mathbf{e}}}^{s}
\]
for all $i,j\in\mathcal{V}$. Thus $G\left(\mathcal{V},\mathcal{E}\right)$
is strongly connected if and only if for all $i,j\in\mathcal{V}$
there exists a positive integer $k$ such that $\left(A^{(s)}\right)_{i,j}^{k}>0$.

A $q\times q$ real matrix $A=(A_{i,j})$ is called \textit{non-negative}
and we write $A\geq0$ if $A_{i,j}\geq0$ for all $1\leq i,j\leq q$.
If $A_{i,j}>0$ holds for all indices $i,j$ then $A$ is called \textit{positive}
and we write $A>0$. For matrices $A$ and $B$ we write $A\geq B$
if $A-B\geq0$ and similarly we write $A>B$ if $A-B>0$. Similar
definitions and notations apply to vectors in $\mathbb{R}^{q}$.

A non-negative matrix $A\geq0$ is called \textit{irreducible} if
for every $1\leq i,j\leq q$ indices there exists a positive integer
$k$ such that $\left(A^{k}\right)_{i,j}>0$. We note that $k$ can
be chosen such that $k\leq q$ (see for example \cite[Lem 1.1.2]{Matrices application-Bapat-Raghavan}).
There are several equivalent definitions of irreducible matrices but
this definition is convenient for us. For a GD-IFS $\left\{ S_{e}:e\in\mathcal{E}\right\} $
we have that $A^{(s)}\geq0$ and $A^{(s)}$ is irreducible if and
only if $G\left(\mathcal{V},\mathcal{E}\right)$ is strongly connected.
Recall that $\rho(A)$ denotes the spectral radius of $A$. The following
theorem is the well-known Perron-Frobenius theorem.

\begin{thm}
\label{thm:Perron-Frobenius-Theorem}Let $A\geq0$ be a $q\times q$
irreducible matrix. Then

(i) there exist $y\in\mathbb{R}^{q}$, $y>0$ and $\lambda_{0}\in\mathbb{R}$,
$\lambda_{0}>0$ such that $Ay=\lambda_{0}y$,

(ii) the eigenvalue $\lambda_{0}$ is a simple root of the characteristic
polynomial of $A$,

(iii) $\rho(A)=\lambda_{0}$,

(iv) the only non-negative, nonzero eigenvectors of $A$ are the positive
scalar multiples of $y$.
\end{thm}

For details see \cite[Thm 1.4.4]{Matrices application-Bapat-Raghavan}.

\begin{rem}
\label{Rem: A^s(y)=00003Dy where y=00003DH^s(K_i)}Let $\left\{ S_{e}:e\in\mathcal{E}\right\} $
be a strongly connected GD-IFS with attractor $\left(K_{1},\ldots,K_{q}\right)$
and let $s$ be the similarity dimension of $\left\{ S_{e}:e\in\mathcal{E}\right\} $.
Let $y_{i}=\mathcal{H}^{s}(K_{i})$ and $y^{\intercal}=(y_{1},\ldots,y_{q})$.
By \cite[Proposition 2]{Wang GDA OSC} we have $\mathcal{H}^{s}(K_{e}\cap K_{f})=0$
for $e,f\in\bigcup_{j=1}^{q}\mathcal{E}_{i,j}$, $e\neq f$. Hence
\[
y_{i}=\mathcal{H}^{s}(K_{i})=\sum_{j=1}^{q}\sum_{e\in\mathcal{E}_{i,j}}\mathcal{H}^{s}(K_{e})=\sum_{j=1}^{q}\sum_{e\in\mathcal{E}_{i,j}}r_{e}^{s}\cdot\mathcal{H}^{s}(K_{j})=\sum_{j=1}^{q}A_{i,j}^{(s)}\cdot y_{j}
\]
so
\[
y=A^{(s)}y.
\]
If $\left\{ S_{e}:e\in\mathcal{E}\right\} $ satisfies the OSC then
by Proposition \ref{prop: OSC GDA Wang} $y\in\mathbb{R}^{q}$, $y>0$.
In that case $y$ satisfies Theorem \ref{thm:Perron-Frobenius-Theorem}
with $1=\rho(A^{(s)})=\lambda_{0}$.
\end{rem}

\begin{cor}
\label{cor:RO(A)=00003D1 Cor}Let $A\geq0$ be a $q\times q$ irreducible
matrix. If there exists a non-negative, non-zero vector $u\in\mathbb{R}^{q}$
such that $Au=u$ then $\rho(A)=1$.
\end{cor}

Corollary \ref{cor:RO(A)=00003D1 Cor} follows from Theorem \ref{thm:Perron-Frobenius-Theorem}.

\begin{lem}
\label{lem:A>Bmatrix}Let $A\geq B\geq0$ be $q\times q$ irreducible
matrices such that $A\neq B$. Then $\rho(A)>\rho(B)$.
\end{lem}

Lemma \ref{lem:A>Bmatrix} follows from \cite[5.7.5]{Marcus-Ming Matrix survey}.

It follows from Lemma \ref{lem:A>Bmatrix} that for a GD-IFS $\rho(A^{(s)})$
is strictly decreasing in $s$.

\begin{lem}
\label{lem:simdim drop by eliminating function}Let $\left\{ S_{e}:e\in\mathcal{E}\right\} $
be a strongly connected GD-IFS with similarity dimension $s$ and
let $e_{0}\in\mathcal{E}$ such that $\left\{ S_{e}:e\in\mathcal{E}\setminus\left\{ e_{0}\right\} \right\} $
is a strongly connected GD-IFS with similarity dimension $s_{0}$.
Then $s_{0}<s$. 
\end{lem}

Lemma \ref{lem:simdim drop by eliminating function} follows from
Lemma \ref{lem:A>Bmatrix}.

\section{Proof of Proposition \ref{lem:Egyenlito disjoint lem T=00003Dinfty}\label{sec:Main-Lemma}}

In this section our main goal is to prove Proposition \ref{lem:Egyenlito disjoint lem T=00003Dinfty}
which provides an important tool to cope with the later results. Proposition
\ref{lem:Egyenlito disjoint lem T=00003Dinfty} is essential in the
proofs in Section \ref{sec:Hausdorff-measure-of} and Section \ref{sec:Transformation-groups-of}
and plays the role of separation conditions when no separation condition
is assumed.

Iterating $K=\bigcup_{i=1}^{m}S_{i}(K)$ gives
\begin{equation}
K=\bigcup_{\boldsymbol{\mathbf{i}}\in\mathcal{I}^{k}}S_{\boldsymbol{\mathbf{i}}}(K)=\bigcup_{\boldsymbol{\mathbf{i}}\in\mathcal{I}^{k}}K_{\boldsymbol{\mathbf{i}}}\label{eq:k th level union}
\end{equation}
for every positive integer $k$.

\begin{lem}
\label{lem:disjointing lem}Let $\left\{ S_{i}\right\} _{i=1}^{m}$
be an SS-IFS with attractor $K$ and $t=\dim_{H}(K)$. Then there
exists $\mathcal{J}\subseteq\bigcup_{k=1}^{\infty}\mathcal{I}^{k}$
such that $K_{\boldsymbol{\mathbf{i}}}\cap K_{\boldsymbol{\mathbf{j}}}=\emptyset$
for $\boldsymbol{\mathbf{i}},\boldsymbol{\mathbf{j}}\in\mathcal{J}$,
$\boldsymbol{\mathbf{i}}\neq\boldsymbol{\mathbf{j}}$ and $\mathcal{H}^{t}\left(K\setminus\left(\bigcup_{\boldsymbol{\mathbf{i}}\in\mathcal{J}}K_{\boldsymbol{\mathbf{i}}}\right)\right)=0$.
\end{lem}

\begin{proof}
$\mathcal{H}^{t}(K)<\infty$ by Proposition \ref{prop:.implicitkov}.
Let $\mathcal{A}=\left\{ K_{\boldsymbol{\mathbf{i}}}:\boldsymbol{\mathbf{i}}\in\bigcup_{k=1}^{\infty}\mathcal{I}^{k}\right\} $.
Then $\mathcal{A}$ is a Vitali cover of $K$ and hence Proposition
\ref{prop:Vtali kov} provides a $\mathcal{J}$ with the required
properties.
\end{proof}

\begin{rem}
\label{rem: disjointing rem}In Lemma \ref{lem:disjointing lem} for
a fixed $\delta>0$ we can further assume that $\mathrm{diam}(K_{\boldsymbol{\mathbf{i}}})<\delta$
for every $\boldsymbol{\mathbf{i}}\in\mathcal{J}$ because in the
proof we can take $\mathcal{A}=\left\{ K_{\boldsymbol{\mathbf{i}}}:\boldsymbol{\mathbf{i}}\in\bigcup_{k=N}^{\infty}\mathcal{I}^{k}\right\} $
for $N$ large enough.
\end{rem}

\begin{lem}
\label{lem:Dense group lem}Let $O\in\overline{\mathcal{T}}$ and
$\delta>0$. Then for each $O_{2}\in\overline{\mathcal{T}}$ there
exists $\boldsymbol{\mathbf{j}}\in\bigcup_{k=1}^{\infty}\mathcal{I}^{k}$
such that $\left\Vert O_{2}\circ T_{\boldsymbol{\mathbf{j}}}-O\right\Vert <\delta$.
\end{lem}

\begin{proof}
By Lemma \ref{lem:lem_i2} we can find $\boldsymbol{\mathbf{j}}\in\bigcup_{k=1}^{\infty}\mathcal{I}^{k}$
such that $\left\Vert O_{2}\circ T_{\boldsymbol{\mathbf{j}}}-O\right\Vert =\left\Vert T_{\boldsymbol{\mathbf{j}}}-O_{2}^{-1}\circ O\right\Vert <\delta$.
\end{proof}

\begin{notation}
For $\boldsymbol{\mathbf{i}}=(i_{1},\ldots,i_{k_{1}}),\boldsymbol{\mathbf{j}}=(j_{1},\ldots,j_{k_{2}})\in\bigcup_{k=1}^{\infty}\mathcal{I}^{k}$
let $\boldsymbol{\mathbf{i}}*\boldsymbol{\mathbf{j}}=(i_{1},\ldots,i_{k_{1}},j_{1},\ldots,j_{k_{2}})$.
\end{notation}

In a metric space $(X,d)$ we call a collection $\mathcal{U}$ of
subsets of $X$ a \textit{$\delta$-cover} for some $\delta>0$, if
$\mathcal{U}$ is a cover of $X$ and $\mathrm{diam}(U)<\delta$ for
every $U\in\mathcal{U}$.\\

\noindent \textit{Proof of Proposition \ref{lem:Egyenlito disjoint lem T=00003Dinfty}.}
Since $\overline{\mathcal{T}}$ is compact there exists a finite open
$\frac{\delta}{2}$-cover $\left\{ U_{i}\right\} _{i=1}^{q}$ of $\overline{\mathcal{T}}$.
Let $\mathcal{V}=\left\{ 1,\ldots,q\right\} $ and for every $i\in\mathcal{V}$
fix $O_{i}\in U_{i}\bigcap\overline{\mathcal{T}}$. By virtue of Lemma
\ref{lem:Dense group lem}, for each $i\in\mathcal{V}$ we find $\boldsymbol{\mathbf{j}}_{i}\in\bigcup_{k=1}^{\infty}\mathcal{I}^{k}$
such that $\left\Vert O_{i}\circ T_{\boldsymbol{\mathbf{j}}_{i}}-O\right\Vert <\frac{\delta}{2}$.
So for every $\boldsymbol{\mathbf{i}}\in\bigcup_{k=1}^{\infty}\mathcal{I}^{k}$
there exists $i\in\mathcal{V}$ such that $T_{\boldsymbol{\mathbf{i}}}\in U_{i}$.
Then $\left\Vert T_{\boldsymbol{\mathbf{i}}}-O_{i}\right\Vert <\frac{\delta}{2}$
and $\left\Vert O_{i}\circ T_{\boldsymbol{\mathbf{j}}_{i}}-O\right\Vert <\frac{\delta}{2}$,
thus
\begin{equation}
\left\Vert T_{\boldsymbol{\mathbf{i}}*\boldsymbol{\mathbf{j}}_{i}}-O\right\Vert <\delta.\label{eq:T_iekj_i<delta}
\end{equation}

By Proposition \ref{prop:.implicitkov} $\mathcal{H}^{t}(K)<\infty$.
Let $\mathcal{J}\subseteq\bigcup_{k=1}^{\infty}\mathcal{I}^{k}$ be
the set provided by Lemma \ref{lem:disjointing lem}. We define a
sequence of sets $\mathcal{I}_{1},\mathcal{I}_{2},\ldots\subseteq\bigcup_{k=1}^{\infty}\mathcal{I}^{k}$
inductively. Let $\mathcal{I}_{1}=\mathcal{J}$. Given $\mathcal{I}_{n}$
is defined we define $\mathcal{I}_{n+1}$ as follows. For each $\boldsymbol{\mathbf{i}}\in\mathcal{I}_{n}$
we define a set $\mathcal{I}_{n+1,\boldsymbol{\mathbf{i}}}$. If $\left\Vert T_{\boldsymbol{\mathbf{i}}}-O\right\Vert <\delta$
then let $\mathcal{I}_{n+1,\boldsymbol{\mathbf{i}}}=\left\{ \boldsymbol{\mathbf{i}}\right\} $.
If $\left\Vert T_{\boldsymbol{\mathbf{i}}}-O\right\Vert \geq\delta$
then $T_{\boldsymbol{\mathbf{i}}}\in U_{i}$ for some $i\in\mathcal{V}$
and $\left\{ K_{\boldsymbol{\mathbf{i}}*\boldsymbol{\mathbf{j}}}:\boldsymbol{\mathbf{j}}\in\bigcup_{k=1}^{\infty}\mathcal{I}^{k},K_{\boldsymbol{\mathbf{i}}*\boldsymbol{\mathbf{j}}}\cap K_{\boldsymbol{\mathbf{i}}*\boldsymbol{\mathbf{j}}_{i}}=\emptyset\right\} $
is a Vitali cover of $K_{\boldsymbol{\mathbf{i}}}\setminus K_{\boldsymbol{\mathbf{i}}*\boldsymbol{\mathbf{j}}_{i}}$,
hence by Proposition \ref{prop:Vtali kov} there exists $\mathcal{J}_{n+1,\boldsymbol{\mathbf{i}}}\subseteq\left\{ \boldsymbol{\mathbf{j}}:\boldsymbol{\mathbf{j}}\in\bigcup_{k=1}^{\infty}\mathcal{I}^{k},K_{\boldsymbol{\mathbf{i}}*\boldsymbol{\mathbf{j}}}\cap K_{\boldsymbol{\mathbf{i}}*\boldsymbol{\mathbf{j}}_{i}}=\emptyset\right\} $
such that $K_{\boldsymbol{\mathbf{i}}*\boldsymbol{\mathbf{i}}_{1}}\cap K_{\boldsymbol{\mathbf{i}}*\boldsymbol{\mathbf{i}}_{2}}=\emptyset$
for $\boldsymbol{\mathbf{i}}_{1},\boldsymbol{\mathbf{i}}_{2}\in\mathcal{J}_{n+1,\boldsymbol{\mathbf{i}}}$,
$\boldsymbol{\mathbf{i}}_{1}\neq\boldsymbol{\mathbf{i}}_{2}$, and
$\mathcal{H}^{t}\left(\left(K_{\boldsymbol{\mathbf{i}}}\setminus K_{\boldsymbol{\mathbf{i}}*\boldsymbol{\mathbf{j}}_{i}}\right)\setminus\left(\bigcup_{\boldsymbol{\mathbf{j}}\in\mathcal{J}_{n+1,\boldsymbol{\mathbf{i}}}}K_{\boldsymbol{\mathbf{i}}*\boldsymbol{\mathbf{j}}}\right)\right)=0$.
Then let $\mathcal{I}_{n+1,\boldsymbol{\mathbf{i}}}=\left\{ \boldsymbol{\mathbf{i}}*\boldsymbol{\mathbf{j}}_{i}\right\} \bigcup\left\{ \boldsymbol{\mathbf{i}}*\boldsymbol{\mathbf{j}}:\boldsymbol{\mathbf{j}}\in\mathcal{J}_{n+1,\boldsymbol{\mathbf{i}}}\right\} $
and let $\mathcal{I}_{n+1}=\bigcup_{\boldsymbol{\mathbf{i}}\in\mathcal{I}_{n}}\mathcal{I}_{n+1,\boldsymbol{\mathbf{i}}}$.

Now we define $\mathcal{I}_{\infty}=\bigcap_{n_{1}=1}^{\infty}\bigcup_{n_{2}=n_{1}}^{\infty}\mathcal{I}_{n_{2}}$.
Clearly $K_{\boldsymbol{\mathbf{i}}}\cap K_{\boldsymbol{\mathbf{j}}}=\emptyset$
for $\boldsymbol{\mathbf{i}},\boldsymbol{\mathbf{j}}\in\mathcal{I}_{\infty}$,
$\boldsymbol{\mathbf{i}}\neq\boldsymbol{\mathbf{j}}$. If $\boldsymbol{\mathbf{i}}\in\mathcal{I}_{n}$
and $\left\Vert T_{\boldsymbol{\mathbf{i}}}-O\right\Vert \geq\delta$
then $\boldsymbol{\mathbf{i}}\notin\mathcal{I}_{n+l}$ for every positive
integer $l$, hence $\boldsymbol{\mathbf{i}}\notin\mathcal{I}_{\infty}$.
So $\left\Vert T_{\boldsymbol{\mathbf{i}}}-O\right\Vert <\delta$
for all $\boldsymbol{\mathbf{i}}\in\mathcal{I}_{\infty}$. Let $r_{\mathrm{\mathrm{min}}}=\min\left\{ r_{\boldsymbol{\mathbf{j}}_{i}}:i\in\mathcal{V}\right\} >0$.
Clearly
\begin{equation}
\mathcal{H}^{t}\left(K\setminus\left(\bigcup_{\boldsymbol{\mathbf{i}}\in\mathcal{I}_{n}}K_{\boldsymbol{\mathbf{i}}}\right)\right)=0\label{eq:Egyenlito lem 1}
\end{equation}
for every positive integer $n$. For $\boldsymbol{\mathbf{i}}\in\mathcal{I}_{n}$
such that $\left\Vert T_{\boldsymbol{\mathbf{i}}}-O\right\Vert \geq\delta$
and $T_{\boldsymbol{\mathbf{i}}}\in U_{i}$ for some $i\in\mathcal{V}$
(if there are more than one such $i$ then we choose the one that
was used above to define the sequence $\mathcal{J}_{n+1,\boldsymbol{\mathbf{i}}}$)
we have that $\left\{ \boldsymbol{\mathbf{j}}:\boldsymbol{\mathbf{i}}*\boldsymbol{\mathbf{j}}\in\mathcal{I}_{n+1},\left\Vert T_{\boldsymbol{\mathbf{i}}*\boldsymbol{\mathbf{j}}}-O\right\Vert \geq\delta\right\} \subseteq\mathcal{J}_{n+1,\boldsymbol{\mathbf{i}}}$
and $\mathcal{H}^{t}(K_{\boldsymbol{\mathbf{i}}*\boldsymbol{\mathbf{j}}_{i}})=r_{\boldsymbol{\mathbf{j}}_{i}}^{t}\mathcal{H}^{t}(K_{\boldsymbol{\mathbf{i}}})\geq r_{\mathrm{\min}}^{t}\mathcal{H}^{t}(K_{\boldsymbol{\mathbf{i}}})$,
and in the mean time $\left\Vert T_{\boldsymbol{\mathbf{i}}*\boldsymbol{\mathbf{j}}_{i}}-O\right\Vert <\delta$
by (\ref{eq:T_iekj_i<delta}). Therefore $\mathcal{I}_{n+1}\setminus\mathcal{I}_{\infty}\subseteq\bigcup_{\boldsymbol{\mathbf{i}}\in\mathcal{I}_{n}\setminus\mathcal{I}_{\infty}}\left\{ \boldsymbol{\mathbf{i}}*\boldsymbol{\mathbf{j}}:\boldsymbol{\mathbf{j}}\in\mathcal{J}_{n+1,\boldsymbol{\mathbf{i}}}\right\} $
and 
\[
\begin{aligned}\mathcal{H}^{t}\left(\bigcup_{\boldsymbol{\mathbf{i}}\in\mathcal{I}_{n+1}\setminus\mathcal{I}_{\infty}}K_{\boldsymbol{\mathbf{i}}}\right) & \leq\sum_{\boldsymbol{\mathbf{i}}\in\mathcal{I}_{n}\setminus\mathcal{I}_{\infty}}\sum_{\boldsymbol{\mathbf{j}}\in\mathcal{J}_{n+1,\boldsymbol{\mathbf{i}}}}\mathcal{H}^{t}(K_{\boldsymbol{\mathbf{i}}*\boldsymbol{\mathbf{j}}})\leq\sum_{\boldsymbol{\mathbf{i}}\in\mathcal{I}_{n}\setminus\mathcal{I}_{\infty}}\left(\mathcal{H}^{t}(K_{\boldsymbol{\mathbf{i}}})-r_{\mathrm{min}}^{t}\mathcal{H}^{t}(K_{\boldsymbol{\mathbf{i}}})\right)\\
 & =\sum_{\boldsymbol{\mathbf{i}}\in\mathcal{I}_{n}\setminus\mathcal{I}_{\infty}}(1-r_{\mathrm{min}}^{t})\cdot\mathcal{H}^{t}(K_{\boldsymbol{\mathbf{i}}})=(1-r_{\mathrm{min}}^{t})\cdot\mathcal{H}^{t}\left(\bigcup_{\boldsymbol{\mathbf{i}}\in\mathcal{I}_{n}\setminus\mathcal{I}_{\infty}}K_{\boldsymbol{\mathbf{i}}}\right).
\end{aligned}
\]
Hence $\mathcal{H}^{t}\left(\bigcup_{\boldsymbol{\mathbf{i}}\in\mathcal{I}_{n+1}\setminus\mathcal{I}_{\infty}}K_{\boldsymbol{\mathbf{i}}}\right)\leq(1-r_{\mathrm{min}}^{t})^{n}\cdot\mathcal{H}^{t}\left(\bigcup_{\boldsymbol{\mathbf{i}}\in\mathcal{I}_{1}\setminus\mathcal{I}_{\infty}}K_{\boldsymbol{\mathbf{i}}}\right)$
for all $n\in\mathbb{N}$ and combined with (\ref{eq:Egyenlito lem 1})
we get that $\mathcal{H}^{t}\left(\bigcup_{\boldsymbol{\mathbf{i}}\in\mathcal{I}_{n+1}\bigcap\mathcal{I}_{\infty}}K_{\boldsymbol{\mathbf{i}}}\right)\geq\left(1-(1-r_{\mathrm{min}}^{t})^{n}\right)\cdot\mathcal{H}^{t}(K)$.
Thus $\mathcal{H}^{t}\left(\bigcup_{\boldsymbol{\mathbf{i}}\in\mathcal{I}_{\infty}}K_{\boldsymbol{\mathbf{i}}}\right)\geq\mathcal{H}^{t}(K)$
and so $\mathcal{H}^{t}\left(K\setminus\left(\bigcup_{\boldsymbol{\mathbf{i}}\in\mathcal{I}_{\infty}}K_{\boldsymbol{\mathbf{i}}}\right)\right)=0$.$\hfill\square$

\begin{cor}
\label{lem:Egyenlito disjoint lem}Let $\left\{ S_{i}\right\} _{i=1}^{m}$
be an SS-IFS with attractor $K$ and let $t=\dim_{H}(K)$. Assume
that $\mathcal{T}$ is a finite group and let $O\in\mathcal{T}$ be
arbitrary. Then there exists $\mathcal{I}_{\infty}\subseteq\bigcup_{k=1}^{\infty}\mathcal{I}^{k}$
such that $T_{\boldsymbol{\mathbf{i}}}=O$ for all $\boldsymbol{\mathbf{i}}\in\mathcal{I}_{\infty}$,
$K_{\boldsymbol{\mathbf{i}}}\cap K_{\boldsymbol{\mathbf{j}}}=\emptyset$
for $\boldsymbol{\mathbf{i}},\boldsymbol{\mathbf{j}}\in\mathcal{I}_{\infty}$,
$\boldsymbol{\mathbf{i}}\neq\boldsymbol{\mathbf{j}}$ and $\mathcal{H}^{t}\left(K\setminus\left(\bigcup_{\boldsymbol{\mathbf{i}}\in\mathcal{I}_{\infty}}K_{\boldsymbol{\mathbf{i}}}\right)\right)=0$.
\end{cor}

\begin{proof}
Let $q=\left|\mathcal{T}\right|$, $\mathcal{T}=\left\{ O_{1},\ldots,O_{q}\right\} $,
$\mathcal{V}=\left\{ 1,\ldots,q\right\} $ and $O=O_{i}$ for some
$i\in\mathcal{V}$. Then let $\delta=\min_{j\in\mathcal{V},j\neq i}\left\Vert O_{j}-O\right\Vert >0$.
By Proposition \ref{lem:Egyenlito disjoint lem T=00003Dinfty} there
is an $\mathcal{I}_{\infty}$ such that $\left\Vert T_{\boldsymbol{\mathbf{i}}}-O\right\Vert <\delta$
and so $T_{\boldsymbol{\mathbf{i}}}=O$ for all $\boldsymbol{\mathbf{i}}\in\mathcal{I}_{\infty}$.
\end{proof}

\begin{prop}
\label{lem:Sok reszes Egyenlito disjoint lem}Let $\left\{ S_{i}\right\} _{i=1}^{m}$
be an SS-IFS with attractor $K$ and let $t=\dim_{H}(K)$. Let $O\in\overline{\mathcal{T}}$
be arbitrary and let $\boldsymbol{\mathbf{i}}_{1},\ldots,\boldsymbol{\mathbf{i}}_{n}\in\bigcup_{k=1}^{\infty}\mathcal{I}^{k}$
be such that $\bigcup_{i=1}^{n}K_{\boldsymbol{\mathbf{i}}_{i}}$ is
a disjoint union and let $\delta>0$. Then there exists $\mathcal{I}_{\infty}\subseteq\bigcup_{k=1}^{\infty}\mathcal{I}^{k}$
such that $\left\Vert T_{\boldsymbol{\mathbf{i}}}-O\right\Vert <\delta$
for all $\boldsymbol{\mathbf{i}}\in\mathcal{I}_{\infty}$, with $\bigcup_{\boldsymbol{\mathbf{i}}\in\mathcal{I}_{\infty}}\bigcup_{i=1}^{n}K_{\boldsymbol{\mathbf{i}}*\boldsymbol{\mathbf{i}}_{i}}$
a disjoint union and $\mathcal{H}^{t}\left(K\setminus\left(\bigcup_{\boldsymbol{\mathbf{i}}\in\mathcal{I}_{\infty}}\bigcup_{i=1}^{n}K_{\boldsymbol{\mathbf{i}}*\boldsymbol{\mathbf{i}}_{i}}\right)\right)=0$.
\end{prop}

The proof of Proposition \ref{lem:Sok reszes Egyenlito disjoint lem}
is similar to the proof of Proposition \ref{lem:Egyenlito disjoint lem T=00003Dinfty}
with the difference that if we have $\boldsymbol{\mathbf{i}}\in\mathcal{I}_{n}$
at a level such that $\left\Vert T_{\boldsymbol{\mathbf{i}}}-O\right\Vert <\delta$
then we keep the pieces $K_{\boldsymbol{\mathbf{i}}*\boldsymbol{\mathbf{i}}_{i}}$
from the next level on and again cover the rest of $K_{\boldsymbol{\mathbf{i}}}$
on the next level.

\section{Iterated function systems with finite transformation groups\label{sec:Iterated-function-sytems}}

In this section we deal with the case when $\mathcal{T}$ is finite.
First, using a natural construction of a GD-IFS we verify Theorem
\ref{cor: SS Lin im is GDA}. Then we prove Theorem \ref{thm:dimension drop SS}.\\

\noindent \textit{Proof of Theorem \ref{cor: SS Lin im is GDA}.}
We need to construct a directed graph $G\left(\mathcal{V},\mathcal{E}\right)$
and a GD-IFS $\left\{ S_{e}:e\in\mathcal{E}\right\} $ that satisfies
the theorem. Let $\mathcal{V}$ be the set $\left\{ 1,2,\ldots,q\right\} $.
For $i,j\in\mathcal{V}$ and for $n\in\mathcal{I}$ we draw a directed
edge $e_{i,j}^{n}$ from $i$ to $j$ if $O_{i}\circ T_{n}=O_{j}$.
Then let $\mathcal{E}=\left\{ e_{i,j}^{n}:i,j\in\mathcal{V},n\in\mathcal{I},O_{i}\circ T_{n}=O_{j}\right\} $.

For $i,j\in\mathcal{V}$ and $n\in\mathcal{I}$ such that $O_{i}\circ T_{n}=O_{j}$,
i.e. $e_{i,j}^{n}=e\in\mathcal{E}$, we write $v_{e}=v_{e_{i,j}^{n}}=L\circ O_{i}(v_{n})$,
$r_{e}=r_{e_{i,j}^{n}}=r_{n}$ and let $S_{e}:\mathbb{R}^{d_{2}}\longrightarrow\mathbb{R}^{d_{2}}$
be the map
\begin{equation}
S_{e}(x)=S_{e_{i,j}^{n}}(x)=r_{e}\cdot x+v_{e}\label{eq:!!!S_e def-homothecy}
\end{equation}
Let $\left\{ S_{e}:e\in\mathcal{E}\right\} $ be the GD-IFS on the
graph $G\left(\mathcal{V},\mathcal{E}\right)$. Since $K=\bigcup_{n=1}^{m}S_{n}(K)$,
for $i\in\mathcal{V}$,
\[
\begin{aligned}L\left(O_{i}(K)\right) & =\bigcup_{n=1}^{m}L\circ O_{i}\circ S_{n}(K)=\bigcup_{n=1}^{m}r_{n}\cdot L\circ O_{i}\circ T_{n}(K)+L\circ O_{i}(v_{n})\\
 & =\bigcup_{n=1}^{m}\bigcup_{j\in\mathcal{V},O_{i}\circ T_{n}=O_{j}}r_{n}\cdot L\circ O_{j}(K)+L\circ O_{i}(v_{n})=\bigcup_{j=1}^{q}\bigcup_{e\in\mathcal{E}_{i,j}}S_{e}\left(L\circ O_{j}(K)\right)
\end{aligned}
\]
and this shows that the $q$-tuple $\left(L\circ O_{1}(K),\ldots,L\circ O_{q}(K)\right)$
is the attractor of $\left\{ S_{e}:e\in\mathcal{E}\right\} $.

Let us show that the the graph $G\left(\mathcal{V},\mathcal{E}\right)$
is strongly connected. Let $i,j\in\mathcal{V}$ be arbitrary. Then
$O_{i}^{-1}\circ O_{j}\in\mathcal{T}$ and since $\mathcal{T}$ is
generated by the transformations $\left\{ T_{i}\right\} _{i=1}^{m}$
and each $T_{i}$ has finite order there exists $\boldsymbol{\mathbf{i}}=(i_{1},\ldots,i_{k})\in\mathcal{I}^{k}$
such that $T_{\boldsymbol{\mathbf{i}}}=T_{i_{1}}\circ\ldots\circ T_{i_{k}}=O_{i}^{-1}\circ O_{j}$.
Let $\boldsymbol{\mathbf{j}}=(j_{1},\ldots,j_{k},j_{k+1})\in\mathcal{V}^{k+1}$
such that $j_{1}=i$ and $O_{j_{n}}\circ T_{i_{n}}=O_{j_{n+1}}$ for
$1\leq n\leq k$. This shows that there exists a $k$ step long directed
path from $i$ to $j$, that visits vertices $i=j_{1},\ldots,j_{k},j_{k+1}=j$
in order. So the graph $G\left(\mathcal{V},\mathcal{E}\right)$ is
strongly connected.

Let $u=(1,\ldots,1)^{\intercal}\in\mathbb{R}^{q}$ be the vector with
each coordinate $1$. For the GD-IFS $\left\{ S_{e}:e\in\mathcal{E}\right\} $
the matrix $A^{(s)}$ is defined as in (\ref{eq:A^s def}) $A_{i,j}^{(s)}=\sum_{e\in\mathcal{E}_{i,j}}r_{e}^{s}$,
hence the $i$th coordinate of the vector $A^{(s)}u$ is
\[
\left(A^{(s)}u\right)_{i}=\sum_{j=1}^{q}A_{i,j}^{(s)}=\sum_{j=1}^{q}\sum_{e\in\mathcal{E}_{i,j}}r_{e}^{s}=\sum_{j=1}^{q}\sum_{n=1}^{m}\sum_{e_{i,j}^{n}\in\mathcal{E}_{i,j}}r_{e_{i,j}^{n}}^{s}=\sum_{n=1}^{m}r_{n}^{s}=1
\]
using that $s$ is the similarity dimension of the SS-IFS $\left\{ S_{i}\right\} _{i=1}^{m}$.
So $u$ is a non-negative, non-zero eigenvector of the irreducible
matrix $A^{(s)}$ with eigenvalue $1$. Thus $\rho(A^{(s)})=1$ by
Corollary \ref{cor:RO(A)=00003D1 Cor} and hence the similarity dimension
of the GD-IFS $\left\{ S_{e}:e\in\mathcal{E}\right\} $ is $s$.

Let $y_{i}=\mathcal{H}^{s}\left(L\circ O_{i}(K)\right)$ and $y^{\intercal}=(y_{1},\ldots,y_{q})$.
If the GD-IFS $\left\{ S_{e}:e\in\mathcal{E}\right\} $ does not satisfy
the OSC then $\mathcal{H}^{s}\left(L\circ O_{1}(K)\right)=\ldots=\mathcal{H}^{s}\left(L\circ O_{q}(K)\right)=0$
by Proposition \ref{prop: OSC GDA Wang}. If the GD-IFS $\left\{ S_{e}:e\in\mathcal{E}\right\} $
satisfies the OSC then $0<\mathcal{H}^{s}(L\circ O_{i}(K))<\infty$
for each $i\in\mathcal{V}$ by Proposition \ref{prop: OSC GDA Wang},
hence $y\in\mathbb{R}^{q}$, $y>0$ and by Remark \ref{Rem: A^s(y)=00003Dy where y=00003DH^s(K_i)}
$A^{(s)}y=y$. So $y$ is a positive scalar multiple of $u$ by Theorem
\ref{thm:Perron-Frobenius-Theorem} \textit{(iv)}. So $y=\mathcal{H}^{s}\left(L\circ O_{i}(K)\right)\cdot u$
for each $i\in\mathcal{V}$ and hence $\mathcal{H}^{s}\left(L\circ O_{1}(K)\right)=\ldots=\mathcal{H}^{s}\left(L\circ O_{q}(K)\right)$.$\hfill\square$\\
\\
Theorem \ref{thm:dimension drop SS} states that we can always find
a projection such that the dimension drops under the image of the
projection. We show this by finding a projection where exact overlapping
occurs.\\

\noindent \textit{Proof of Theorem \ref{thm:dimension drop SS}.}
We can assume that $l=d-1$ because if $M\in G_{d,d-1}$ such that
$\dim_{H}\left(\Pi_{M}(K)\right)<s$ and $N$ is a subspace contained
in $M$ then $\Pi_{N}=\Pi_{N}\circ\Pi_{M}$ hence $\dim_{H}\left(\Pi_{N}(K)\right)<s$.
Let $\mathcal{T}=\left\{ O_{1},\ldots,O_{q}\right\} $ where $q=\left|\mathcal{T}\right|$
and let $\mathcal{V}=\left\{ 1,2,\ldots,q\right\} $. We may assume
that $T_{1}=T_{2}=Id_{\mathbb{R}^{d}}$ because if we iterate the
IFS $q$ times then we obtain the SS-IFS $\left\{ S_{\boldsymbol{\mathbf{i}}}:\boldsymbol{\mathbf{i}}\in\mathcal{I}^{q}\right\} $.
The similarity dimension of this new SS-IFS is $s$, the attractor
of it is $K$ and the transformation group of it is a subgroup of
$\mathcal{T}$, hence is finite. Since $q$ is the order of $\mathcal{T}$
it follows that $T_{1}^{q}=T_{2}^{q}=Id_{\mathbb{R}^{d}}$. So taking
the new IFS after relabeling we have that $T_{1}=T_{2}=Id_{\mathbb{R}^{d}}$.
We can further assume that $r_{1}=r_{2}$ because if we iterate the
IFS, we obtain the SS-IFS $\left\{ S_{\boldsymbol{\mathbf{i}}}:\boldsymbol{\mathbf{i}}\in\mathcal{I}^{2}\right\} $
and again the similarity dimension, the attractor and the finiteness
of the transformation group do not change. Then $r_{1}\cdot r_{2}=r_{2}\cdot r_{1}$,
$T_{1}\circ T_{2}=Id_{\mathbb{R}^{d}}\circ Id_{\mathbb{R}^{d}}=T_{2}\circ T_{1}$.
So taking the new IFS after relabeling we have that $T_{1}=T_{2}=Id_{\mathbb{R}^{d}}$
and $r_{1}=r_{2}$.

So $K_{1}=S_{1}(K)$ is a translate of $K_{2}=S_{2}(K)$. Let $v$
be the translation vector such that $K_{1}=K_{2}+v$. Let $M$ be
the orthogonal direct complement of $v$ (if $v=0$ then $M\in G_{d,d-1}$
can be arbitrary). Then $\Pi_{M}(K_{1})=\Pi_{M}(K_{2})$. Let $L=\Pi_{M}:\mathbb{R}^{d}\longrightarrow\mathbb{R}^{d-1}$.
Then let $G\left(\mathcal{V},\mathcal{E}\right)$ be the graph, $\left\{ S_{e}:e\in\mathcal{E}\right\} $
be the GD-IFS that is constructed in the proof of Theorem \ref{cor: SS Lin im is GDA}
and for $i,j\in\mathcal{V}$ and for $n\in\mathcal{I}$, such that
$O_{i}\circ T_{n}=O_{j}$, let $e_{i,j}^{n}$ as in the proof. Let
$i\in\mathcal{V}$ such that $O_{i}=Id_{\mathbb{R}^{d}}$. Then $e_{i,i}^{1}$
and $e_{i,i}^{2}$ are loops in $G\left(\mathcal{V},\mathcal{E}\right)$
and
\[
\begin{aligned}S_{e_{i,i}^{1}}(\Pi_{M}(K)) & =r_{1}\cdot\Pi_{M}(K)+\Pi_{M}(v_{1})=r_{1}\cdot\Pi_{M}(K)+\Pi_{M}(v_{1}-v)\\
 & =r_{2}\cdot\Pi_{M}(K)+\Pi_{M}(v_{2})=S_{e_{i,i}^{2}}(\Pi_{M}(K)).
\end{aligned}
\]
So if we take $\left\{ S_{e}:e\in\mathcal{E}\setminus\left\{ e_{0}\right\} \right\} $
with $e_{0}=e_{i,i}^{2}$ then $\left\{ S_{e}:e\in\mathcal{E}\setminus\left\{ e_{0}\right\} \right\} $
is a strongly connected GD-IFS with attractor $\left(\Pi_{M}\circ O_{1}(K),\ldots,\Pi_{M}\circ O_{q}(K)\right)$.
So by Lemma \ref{lem:simdim drop by eliminating function} the similarity
dimension of $\left\{ S_{e}:e\in\mathcal{E}\setminus\left\{ e_{0}\right\} \right\} $
is strictly smaller than $s$. Hence $\dim_{H}\left(\Pi_{M}(K)\right)<s$
by Lemma \ref{prop:finite measure GDA}.$\hfill\square$

\section{Hausdorff measure of the orbits\label{sec:Hausdorff-measure-of}}

In this section we deal with the general results when we have no restriction
on $\mathcal{T}$ and our main aim is to prove Theorem \ref{thm: Egynlo lin Im}.
At the end of this section we conclude Corollary \ref{cor: SS almost disjoint A B}
from Theorem \ref{thm: Egynlo lin Im}.

First we prove Proposition \ref{lem:Lin neighbourhood lem} that says
the Hausdorff measure of linear images of $K$ is upper semi-continuous
in the linear maps. This observation is essential in the proof of
Theorem \ref{thm: Egynlo lin Im}.\\

\begin{prop}
\label{prop:upersemicont}Let $\left\{ S_{i}\right\} _{i=1}^{m}$
be an SS-IFS with attractor $K\subseteq\mathbb{R}^{d}$, let $t=\dim_{H}(K)$
and $L:\mathbb{R}^{d}\longrightarrow\mathbb{R}^{d_{2}}$ be a linear
map. Then for every $\varepsilon>0$ there exists $\delta>0$ such
that for every linear map $L_{2}:\mathbb{R}^{d}\longrightarrow\mathbb{R}^{d_{2}}$
with $\left\Vert L-L_{2}\right\Vert <\delta$ we have that $\mathcal{H}^{t}(L_{2}(K))\leq\mathcal{H}_{\infty}^{t}(L(K))+\varepsilon$.\end{prop}
\begin{proof}
\noindent It is enough to verify the proposition for $0<\varepsilon<1$.
We may assume that $\mathcal{H}^{t}(K)>0$ otherwise $\mathcal{H}^{t}(L_{2}(K))=0$.
By Proposition \ref{prop:.implicitkov} $\mathcal{H}^{t}(K)<\infty$
hence $K$ is a $t$-set. Since $K$ is compact there exists $R>0$
such that $K$ is contained in $B(0,R)$. Since $L(K)$ is compact
it follows that $\mathcal{H}_{\infty}^{t}(L(K))<\infty$. Thus there
exists a finite open cover $\mathcal{C}$ of $L(K)$ such that
\begin{equation}
\sum_{C\in\mathcal{C}}\mathrm{diam}(C)^{t}\leq\mathcal{H}_{\infty}^{t}(L(K))+\varepsilon\label{eq:Grass neig lem eq0}
\end{equation}
and $0<\mathrm{diam}(C)<\infty$. Let $d_{\mathrm{max}}=\max\left\{ \mathrm{diam}(C):C\in\mathcal{C}\right\} $.
Because $\mathcal{C}$ is a cover of $L(K)$ and $K\subseteq B(0,R)$
it follows that $K\subseteq\bigcup_{C\in\mathcal{C}}L^{-1}(C)\cap B(0,R)$
and $L^{-1}(C)\cap B(0,R)$ is a bounded set such that $\mathrm{diam}\left(L\left(L^{-1}(C)\cap B(0,R)\right)\right)=\mathrm{diam}(C)>0$
for each $C\in\mathcal{C}$. Hence for each $C\in\mathcal{C}$ we
can find $\delta_{C}>0$ such that if $L_{2}:\mathbb{R}^{d}\longrightarrow\mathbb{R}^{d_{2}}$
is a linear map with $\left\Vert L-L_{2}\right\Vert <\delta_{C}$
then $\mathrm{diam}\left(L_{2}\left(L^{-1}(C)\cap B(0,R)\right)\right)\leq\mathrm{diam}(C)\cdot(1+\varepsilon)$.
Let $\delta=\min\left\{ \delta_{C}:C\in\mathcal{C}\right\} /\left\Vert L\right\Vert +1>0$.
So if $\left\Vert L-L_{2}\right\Vert <\delta$ for some linear map
$L_{2}:\mathbb{R}^{d}\longrightarrow\mathbb{R}^{d_{2}}$ and $\left\Vert T-Id_{\mathbb{R}^{d}}\right\Vert <\delta$
for some $T\in\mathbb{O}_{d}$ then

\[
\left\Vert L-L_{2}\circ T\right\Vert =\left\Vert L-L\circ T+L\circ T-L_{2}\circ T\right\Vert <\delta(\left\Vert L\right\Vert +1)=\min\left\{ \delta_{C}:C\in\mathcal{C}\right\} .
\]
Hence
\begin{equation}
\mathrm{diam}\left(L_{2}\circ T\left(L^{-1}(C)\cap B(0,R)\right)\right)\leq\mathrm{diam}(C)\cdot(1+\varepsilon)<2d_{\mathrm{max}}.\label{eq:Lin neigbourhood lem EQ1}
\end{equation}

The lemma will follow if we show that $\mathcal{H}_{\eta}^{t}(L_{2}(K))\leq(1+\varepsilon)^{t}\cdot\left(\mathcal{H}_{\infty}^{t}(L(K))+\varepsilon\right)$
for every $\eta>0$ and linear map $L_{2}:\mathbb{R}^{d}\longrightarrow\mathbb{R}^{d_{2}}$
with $\left\Vert L-L_{2}\right\Vert <\delta$, where $\mathcal{H}_{\eta}^{t}$
denotes Hausdorff pre-measure, used to define Hausdorff measure. Let
$\eta>0$ be fixed, $r_{\mathrm{max}}=\max\left\{ r_{i}:i\in\mathcal{I}\right\} $,
let $k$ be a positive integer such that $r_{\mathrm{max}}^{k}\cdot2d_{\mathrm{max}}<\eta$
and $L_{2}:\mathbb{R}^{d}\longrightarrow\mathbb{R}^{d_{2}}$ be a
linear map with $\left\Vert L-L_{2}\right\Vert <\delta$. Then $K$
is the attractor of the SS-IFS $\left\{ S_{\boldsymbol{\mathbf{i}}}:\boldsymbol{\mathbf{i}}\in\mathcal{I}^{k}\right\} $.
We apply Lemma \ref{lem:Egyenlito disjoint lem T=00003Dinfty} to
the SS-IFS $\left\{ S_{\boldsymbol{\mathbf{i}}}:\boldsymbol{\mathbf{i}}\in\mathcal{I}^{k}\right\} $
with $O=Id_{\mathbb{R}^{d}}$ and $\delta>0$ to obtain $\mathcal{I}_{\infty}\subseteq\bigcup_{k_{2}=1}^{\infty}\mathcal{I}^{k\cdot k_{2}}$
such that $\left\Vert T_{\boldsymbol{\mathbf{i}}}-Id_{\mathbb{R}^{d}}\right\Vert <\delta$
for all $\boldsymbol{\mathbf{i}}\in\mathcal{I}_{\infty}$, $K_{\boldsymbol{\mathbf{i}}}\cap K_{\boldsymbol{\mathbf{j}}}=\emptyset$
for $\boldsymbol{\mathbf{i}},\boldsymbol{\mathbf{j}}\in\mathcal{I}_{\infty}$,
$\boldsymbol{\mathbf{i}}\neq\boldsymbol{\mathbf{j}}$, and $\mathcal{H}^{t}\left(K\setminus\left(\bigcup_{\boldsymbol{\mathbf{i}}\in\mathcal{I}_{\infty}}K_{\boldsymbol{\mathbf{i}}}\right)\right)=0$.
As $K$ is a $t$-set we have by Remark \ref{Rem: Mainlem EQ} that
\begin{equation}
\sum_{\boldsymbol{\mathbf{i}}\in\mathcal{I}_{\infty}}r_{\boldsymbol{\mathbf{i}}}^{t}=1.\label{eq: Grass neigh lem eq1}
\end{equation}
Since $\left\Vert T_{\boldsymbol{\mathbf{i}}}-Id_{\mathbb{R}^{d}}\right\Vert <\delta$
for $\boldsymbol{\mathbf{i}}\in\mathcal{I}_{\infty}$ and $\left\Vert L-L_{2}\right\Vert <\delta$,
it follows from (\ref{eq:Lin neigbourhood lem EQ1}) that
\begin{equation}
\begin{aligned}\mathrm{diam}\left(L_{2}\left(S_{\boldsymbol{\mathbf{i}}}\left(L^{-1}(C)\cap B(0,R)\right)\right)\right) & \leq r_{\boldsymbol{\mathbf{i}}}\cdot\mathrm{diam}\left(L_{2}\left(T_{\boldsymbol{\mathbf{i}}}\left(L^{-1}(C)\cap B(0,R)\right)\right)\right)\\
 & \leq r_{\boldsymbol{\mathbf{i}}}\cdot\mathrm{diam}\left(C\right)\cdot(1+\varepsilon)<r_{\boldsymbol{\mathbf{i}}}\cdot2d_{\mathrm{max}}\leq r_{\mathrm{max}}^{k}\cdot2d_{\mathrm{max}}<\eta.
\end{aligned}
\label{eq:Grass neigh lem eq2}
\end{equation}
Thus $\left\{ L_{2}\left(S_{\boldsymbol{\mathbf{i}}}\left(L^{-1}(C)\cap B(0,R)\right)\right):C\in\mathcal{C},\boldsymbol{\mathbf{i}}\in\mathcal{I}_{\infty}\right\} $
is an $\eta$-cover of $L_{2}\left(\bigcup_{\boldsymbol{\mathbf{i}}\in\mathcal{I}_{\infty}}K_{\boldsymbol{\mathbf{i}}}\right)$
and
\begin{equation}
\begin{aligned}\mathcal{H}_{\eta}^{t}\left(L_{2}\left(\bigcup_{\boldsymbol{\mathbf{i}}\in\mathcal{I}_{\infty}}K_{\boldsymbol{\mathbf{i}}}\right)\right) & \leq\sum_{C\in\mathcal{C}}\sum_{\boldsymbol{\mathbf{i}}\in\mathcal{I}_{\infty}}\mathrm{diam}\left(L_{2}\left(S_{\boldsymbol{\mathbf{i}}}\left(L^{-1}(C)\cap B(0,R)\right)\right)\right)^{t}\\
 & \leq\sum_{C\in\mathcal{C}}\sum_{\boldsymbol{\mathbf{i}}\in\mathcal{I}_{\infty}}\left(r_{\boldsymbol{\mathbf{i}}}\cdot\mathrm{diam}\left(C\right)\cdot(1+\varepsilon)\right)^{t}\\
 & \leq(1+\varepsilon)^{t}\sum_{C\in\mathcal{C}}\mathrm{diam}\left(C\right)^{t}\sum_{\boldsymbol{\mathbf{i}}\in\mathcal{I}_{\infty}}r_{\boldsymbol{\mathbf{i}}}^{t}\leq(1+\varepsilon)^{t}\sum_{C\in\mathcal{C}}\mathrm{diam}\left(C\right)^{t}\\
 & \leq(1+\varepsilon)^{t}\cdot\left(\mathcal{H}_{\infty}^{t}(L(K))+\varepsilon\right)
\end{aligned}
\label{eq:Grass neigh lem eq3}
\end{equation}
where we used (\ref{eq:Grass neig lem eq0}), (\ref{eq: Grass neigh lem eq1})
and (\ref{eq:Grass neigh lem eq2}).

Since $\mathcal{H}^{t}\left(K\setminus\left(\bigcup_{\boldsymbol{\mathbf{i}}\in\mathcal{I}_{\infty}}K_{\boldsymbol{\mathbf{i}}}\right)\right)=0$
it follows  $\mathcal{H}_{\eta}^{t}\left(L_{2}\left(K\setminus\left(\bigcup_{\boldsymbol{\mathbf{i}}\in\mathcal{I}_{\infty}}K_{\boldsymbol{\mathbf{i}}}\right)\right)\right)=0$.
Thus by (\ref{eq:Grass neigh lem eq3})
\[
\mathcal{H}_{\eta}^{t}(L_{2}(K))\leq(1+\varepsilon)^{t}\cdot\left(\mathcal{H}_{\infty}^{t}(L(K))+\varepsilon\right)
\]
which completes the proof.
\end{proof}

\noindent \textit{Proof of Proposition \ref{lem:Lin neighbourhood lem}.}
The proposition is an immediate consequence of Proposition \ref{prop:upersemicont}.$\hfill\square$\\

\noindent \textit{Proof of Theorem \ref{thm: Egynlo lin Im}.} By
Proposition \ref{prop:.implicitkov} $\mathcal{H}^{t}(K)<\infty$
hence $K$ is a $t$-set. Let $\varepsilon>0$ be arbitrary. Let $\delta>0$
such that for every linear map $L_{2}:\mathbb{R}^{d}\longrightarrow\mathbb{R}^{d_{2}}$
with $\left\Vert L-L_{2}\right\Vert <\delta$ we have that $\mathcal{H}^{t}(L_{2}(K))\leq\mathcal{H}^{t}(L(K))+\varepsilon$.
Such a $\delta>0$ exists by Proposition \ref{lem:Lin neighbourhood lem}.
Let $\mathcal{I}_{\infty}$ be the set provided by Proposition \ref{lem:Egyenlito disjoint lem T=00003Dinfty}
for $O^{-1}$ in place of $O$ and $\frac{\delta}{\left\Vert L\right\Vert }$
in place of $\delta$. Then $\left\Vert O\circ T_{\boldsymbol{\mathbf{i}}}-Id_{\mathbb{R}^{d}}\right\Vert =\left\Vert T_{\boldsymbol{\mathbf{i}}}-O^{-1}\right\Vert <\frac{\delta}{\left\Vert L\right\Vert }$
for every $\boldsymbol{\mathbf{i}}\in\mathcal{I}_{\infty}$, hence
$\left\Vert L\circ O\circ T_{\boldsymbol{\mathbf{i}}}-L\right\Vert \leq\left\Vert L\right\Vert \cdot\left\Vert O\circ T_{\boldsymbol{\mathbf{i}}}-Id_{\mathbb{R}^{d}}\right\Vert <\delta$.
So $\mathcal{H}^{t}\left(L\circ O\circ T_{\boldsymbol{\mathbf{i}}}(K)\right)\leq\mathcal{H}^{t}\left(L(K)\right)+\varepsilon$
for every $\boldsymbol{\mathbf{i}}\in\mathcal{I}_{\infty}$, hence
\begin{equation}
\mathcal{H}^{t}\left(L\circ O(S_{\boldsymbol{\mathbf{i}}}(K))\right)=r_{\boldsymbol{\mathbf{i}}}^{t}\cdot\mathcal{H}^{t}\left(L\circ O\circ T_{\boldsymbol{\mathbf{i}}}(K)\right)\leq r_{\boldsymbol{\mathbf{i}}}^{t}\cdot\left(\mathcal{H}^{t}\left(L(K)\right)+\varepsilon\right).\label{eq:Egyen linimeq -01}
\end{equation}
Since $K$ is a $t$-set we have by Remark \ref{Rem: Mainlem EQ}
that
\begin{equation}
\sum_{\boldsymbol{\mathbf{i}}\in\mathcal{I}_{\infty}}r_{\boldsymbol{\mathbf{i}}}^{t}=1.\label{eq:egyen linim eq00}
\end{equation}
It follows that
\[
\sum_{\boldsymbol{\mathbf{i}}\in\mathcal{I}_{\infty}}\mathcal{H}^{t}\left(L\circ O(S_{\boldsymbol{\mathbf{i}}}(K))\right)\leq\sum_{\boldsymbol{\mathbf{i}}\in\mathcal{I}_{\infty}}r_{\boldsymbol{\mathbf{i}}}^{t}\cdot\left(\mathcal{H}^{t}\left(L(K)\right)+\varepsilon\right)=\mathcal{H}^{t}\left(L(K)\right)+\varepsilon
\]
where we have used (\ref{eq:Egyen linimeq -01}) and (\ref{eq:egyen linim eq00}).
Because $\mathcal{H}^{t}\left(L\circ O\left(K\setminus\left(\bigcup_{\boldsymbol{\mathbf{i}}\in\mathcal{I}_{\infty}}S_{\boldsymbol{\mathbf{i}}}(K)\right)\right)\right)=0$
it follows that $\mathcal{H}^{t}\left(L\circ O(K)\right)\leq\mathcal{H}^{t}\left(L(K)\right)+\varepsilon$
for all $\varepsilon>0$. Hence $\mathcal{H}^{t}\left(L\circ O(K)\right)\leq\mathcal{H}^{t}\left(L(K)\right)$.
Replacing $L$ by $L\circ O$ and $O$ by $O^{-1}$, with the same
argument we get that $\mathcal{H}^{t}\left(L(K)\right)=\mathcal{H}^{t}\left(L\circ O\circ O^{-1}(K)\right)\leq\mathcal{H}^{t}\left(L\circ O(K)\right)$.
Thus $\mathcal{H}^{t}\left(L\circ O(K)\right)=\mathcal{H}^{t}\left(L(K)\right)$
and so (\ref{eq:ORBIT_EQ}) holds for $A=K$.

Let $\boldsymbol{\mathbf{i}}\in\mathcal{I}^{k}$ for some $k\in\mathbb{N}$.
Then
\[
\begin{aligned}\mathcal{H}^{t}\left(L\circ O(K_{\boldsymbol{\mathbf{i}}})\right) & =\mathcal{H}^{t}\left(L\circ O(S_{\boldsymbol{\mathbf{i}}}(K))\right)=\mathcal{H}^{t}\left(L\circ O(r_{\boldsymbol{\mathbf{i}}}\cdot T_{\boldsymbol{\mathbf{i}}}(K)+v_{\boldsymbol{\mathbf{i}}})\right)\\
 & =r_{\boldsymbol{\mathbf{i}}}^{t}\cdot\mathcal{H}^{t}\left(L\circ O\circ T_{\boldsymbol{\mathbf{i}}}(K)\right)=\frac{\mathcal{H}^{t}(K_{\boldsymbol{\mathbf{i}}})}{\mathcal{H}^{t}(K)}\cdot\mathcal{H}^{t}\left(L(K)\right)
\end{aligned}
\]
where we used (\ref{eq:ORBIT_EQ}) when $A=K$. So (\ref{eq:ORBIT_EQ})
holds for $A=K_{\boldsymbol{\mathbf{i}}}$, for each $\boldsymbol{\mathbf{i}}\in\mathcal{I}^{k}$,
$k\in\mathbb{N}$.

Let $\mathcal{J}$ be the set provided by Lemma \ref{lem:disjointing lem}.
For every $k\in\mathbb{N}$ let us denote the set $\left\{ \boldsymbol{\mathbf{i}}_{1}*\ldots*\boldsymbol{\mathbf{i}}_{k}:\boldsymbol{\mathbf{i}}_{1},\ldots,\boldsymbol{\mathbf{i}}_{k}\in\mathcal{J}\right\} $
by $\mathcal{J}^{k}$. For $k\in\mathbb{N}$
\begin{equation}
\mathcal{H}^{t}\left(K\setminus\left(\bigcup_{\boldsymbol{\mathbf{i}}\in\mathcal{J}^{k}}K_{\boldsymbol{\mathbf{i}}}\right)\right)=0,\label{eq: egyenlovet-full mertek}
\end{equation}
thus
\[
\mathcal{H}^{t}(K)=\mathcal{H}^{t}\left(\bigcup_{\boldsymbol{\mathbf{i}}\in\mathcal{J}^{k}}K_{\boldsymbol{\mathbf{i}}}\right)=\sum_{\boldsymbol{\mathbf{i}}\in\mathcal{J}^{k}}\mathcal{H}^{t}(K_{\boldsymbol{\mathbf{i}}}).
\]
So
\[
\begin{aligned}\sum_{\boldsymbol{\mathbf{i}}\in\mathcal{J}^{k}}\mathcal{H}^{t}\left(L\circ O(K_{\boldsymbol{\mathbf{i}}})\right) & =\sum_{\boldsymbol{\mathbf{i}}\in\mathcal{J}^{k}}\frac{\mathcal{H}^{t}(K_{\boldsymbol{\mathbf{i}}})}{\mathcal{H}^{t}(K)}\cdot\mathcal{H}^{t}\left(L(K)\right)=\frac{\mathcal{H}^{t}(K)}{\mathcal{H}^{t}(K)}\cdot\mathcal{H}^{t}\left(L(K)\right)\\
 & =\mathcal{H}^{t}\left(L\circ O(K)\right)=\mathcal{H}^{t}\left(L\circ O\left(\bigcup_{\boldsymbol{\mathbf{i}}\in\mathcal{J}^{k}}K_{\boldsymbol{\mathbf{i}}}\right)\right)
\end{aligned}
\]
where we used (\ref{eq:ORBIT_EQ}) for $A=K_{\boldsymbol{\mathbf{i}}}$
and for $A=K$. It follows that $\mathcal{H}^{t}\left(L\circ O(K_{\boldsymbol{\mathbf{i}}})\cap L\circ O(K_{\boldsymbol{\mathbf{j}}})\right)=0$
for $\boldsymbol{\mathbf{i}},\boldsymbol{\mathbf{j}}\in\mathcal{J}^{k}$,
$\boldsymbol{\mathbf{i}}\neq\boldsymbol{\mathbf{j}}$. Hence (\ref{eq:ORBIT_EQ})
holds for $A=\bigcup_{\boldsymbol{\mathbf{i}}\in\mathcal{F}}K_{\boldsymbol{\mathbf{i}}}$
where $\mathcal{F}\subseteq\mathcal{J}^{k}$.

Using (\ref{eq: egyenlovet-full mertek}) and the continuity of measures
it follows that
\begin{equation}
\mathcal{H}^{t}\left(K\setminus\left(\bigcap_{k=1}^{\infty}\bigcup_{\boldsymbol{\mathbf{i}}\in\mathcal{J}^{k}}K_{\boldsymbol{\mathbf{i}}}\right)\right)=0.\label{eq:egynlovet full mertek 222}
\end{equation}

Assume that $A\subseteq\bigcap_{k=1}^{\infty}\bigcup_{\boldsymbol{\mathbf{i}}\in\mathcal{J}^{k}}K_{\boldsymbol{\mathbf{i}}}$
is compact, $\varepsilon>0$ arbitrary and let
\[
F_{k}=\bigcup_{\boldsymbol{\mathbf{i}}\in\mathcal{J}^{k},K_{\boldsymbol{\mathbf{i}}}\cap A\neq\emptyset}K_{\boldsymbol{\mathbf{i}}}.
\]
Then $K\supseteq F_{1}\supseteq F_{2}\supseteq\ldots$ and $A=\bigcap_{k=1}^{\infty}F_{k}$.
Thus there exists $k$ such that $\mathcal{H}^{t}\left(F_{k}\setminus A\right)<\varepsilon$.
Since (\ref{eq:ORBIT_EQ}) holds for $F_{k}$ it follows that
\begin{equation}
\begin{aligned}\mathcal{H}^{t}\left(L\circ O(A)\right) & \leq\mathcal{H}^{t}\left(L\circ O(F_{k})\right)=\frac{\mathcal{H}^{t}(F_{k})}{\mathcal{H}^{t}(K)}\mathcal{H}^{t}\left(L(K)\right)\\
 & =\frac{\mathcal{H}^{t}(A)+\mathcal{H}^{t}(F_{k}\setminus A)}{\mathcal{H}^{t}(K)}\mathcal{H}^{t}\left(L(K)\right)\leq\frac{\mathcal{H}^{t}(A)+\varepsilon}{\mathcal{H}^{t}(K)}\mathcal{H}^{t}\left(L(K)\right)
\end{aligned}
\label{eq:Refeq1}
\end{equation}
and
\begin{equation}
\begin{aligned}\mathcal{H}^{t}\left(L\circ O(A)\right) & \geq\mathcal{H}^{t}\left(L\circ O(F_{k})\right)-\mathcal{H}^{t}\left(L\circ O(F_{k}\setminus A)\right)\\
 & \geq\frac{\mathcal{H}^{t}(F_{k})}{\mathcal{H}^{t}(K)}\mathcal{H}^{t}\left(L(K)\right)-\left\Vert L\circ O\right\Vert ^{t}\cdot\mathcal{H}^{t}\left(F_{k}\setminus A\right)\\
 & \geq\frac{\mathcal{H}^{t}(A)}{\mathcal{H}^{t}(K)}\mathcal{H}^{t}\left(L(K)\right)-\left\Vert L\right\Vert ^{t}\cdot\varepsilon.
\end{aligned}
\label{eq:Refeq2}
\end{equation}
Since $\varepsilon>0$ is arbitrary (\ref{eq:ORBIT_EQ}) holds for
compact $A\subseteq\bigcap_{k=1}^{\infty}\bigcup_{\boldsymbol{\mathbf{i}}\in\mathcal{J}^{k}}K_{\boldsymbol{\mathbf{i}}}$.

Now assume that $A$ is any $\mathcal{H}^{t}$-measurable set and
$\varepsilon>0$ is arbitrary. By (\ref{eq:egynlovet full mertek 222})
\[
\mathcal{H}^{t}\left(A\bigcap\left(\bigcap_{k=1}^{\infty}\bigcup_{\boldsymbol{\mathbf{i}}\in\mathcal{J}^{k}}K_{\boldsymbol{\mathbf{i}}}\right)\right)=\mathcal{H}^{t}(A).
\]
Hence we can find a compact $F\subseteq A\bigcap\left(\bigcap_{k=1}^{\infty}\bigcup_{\boldsymbol{\mathbf{i}}\in\mathcal{J}^{k}}K_{\boldsymbol{\mathbf{i}}}\right)\subseteq A$
such that $\mathcal{H}^{t}(A\setminus F)<\varepsilon$. Using a similar
argument to (\ref{eq:Refeq1}) and (\ref{eq:Refeq2}) we can deduce
that $\mathcal{H}^{t}\left(L\circ O(A)\right)=\frac{\mathcal{H}^{t}(A)}{\mathcal{H}^{t}(K)}\mathcal{H}^{t}\left(L(K)\right)$
because (\ref{eq:ORBIT_EQ}) holds for $F$. Thus (\ref{eq:ORBIT_EQ})
holds for every $\mathcal{H}^{t}$-measurable $A$.

Now let $A\subseteq K$ be arbitrary and let $B$ be a $\mathcal{H}^{t}$-measurable
hull of $A$ such that $A\subseteq B\subseteq K$. By virtue of Lemma
\ref{lem:Hull Lemma} and applying (\ref{eq:ORBIT_EQ}) to $B$ we
get that (\ref{eq:ORBIT_EQ}) holds for $A$.$\hfill\square$\\

\noindent \textit{Proof of Corollary \ref{cor: SS almost disjoint A B}.}
If $\mathcal{H}^{t}(K)=0$ then the statement is trivial, so we can
assume that $\mathcal{H}^{t}(K)>0$. Since $B\subseteq(K\setminus A)\cup(A\cap B)$
and $\mathcal{H}^{t}\left(A\cap B\right)=0$ it is enough to show
that $\mathcal{H}^{t}\left(L(A)\cap L(K\setminus A)\right)=0$. By
Theorem \ref{thm: Egynlo lin Im}
\[
\begin{aligned}\mathcal{H}^{t}\left(L(K)\right) & =\mathcal{H}^{t}\left(L(A)\cup L(K\setminus A)\right)\\
 & =\mathcal{H}^{t}\left(L(A)\right)+\mathcal{H}^{t}\left(L(K\setminus A)\right)-\mathcal{H}^{t}\left(L(A)\cap L(K\setminus A)\right)\\
 & =\frac{\mathcal{H}^{t}(A)}{\mathcal{H}^{t}(K)}\mathcal{H}^{t}\left(L(K)\right)+\frac{\mathcal{H}^{t}(K\setminus A)}{\mathcal{H}^{t}(K)}\mathcal{H}^{t}\left(L(K)\right)-\mathcal{H}^{t}\left(L(A)\cap L(K\setminus A)\right)\\
 & =\mathcal{H}^{t}\left(L(K)\right)-\mathcal{H}^{t}\left(L(A)\cap L(K\setminus A)\right).
\end{aligned}
\]
Hence $\mathcal{H}^{t}\left(L(A)\cap L(K\setminus A)\right)=0$ since
$\mathcal{H}^{t}\left(L(K)\right)<\infty$ by Proposition \ref{prop:.implicitkov}.$\hfill\square$

\section{Transformation groups of dense orbits and Hausdorff measure\label{sec:Transformation-groups-of}}

In this section our main goal is to prove Theorem \ref{cor:Dense group lin image 0}.
First we show that under the assumptions of Theorem \ref{cor:Dense group lin image 0}
every linear image of $K$ is of zero measure. Then we generalise
this for continuously differentiable maps.

\begin{lem}
\label{lem: eta-cover lemma}Let $\mathcal{G}$ be a closed subset
of $\mathbb{O}_{d}$, let $K\subseteq\mathbb{R}^{d}$ be a compact
set, let $L:\mathbb{R}^{d}\longrightarrow\mathbb{R}^{d_{2}}$ be a
linear map and $c>0$ be such that $\mathcal{H}_{\infty}^{t}\left(L\circ O(K)\right)<c$
for every $O\in\mathcal{G}$. Then there exists $\zeta>0$ such that
for every $O\in\mathcal{G}$ there exists a finite open cover $\mathcal{U}$
of $L\circ O(K)$ such that $\sum_{U\in\mathcal{U}}\mathrm{diam}(U)^{t}<c$
and $\min_{U\in\mathcal{U}}\mathrm{diam}(U)>\zeta$.
\end{lem}

\begin{proof}
For every $O\in\mathcal{G}$ we can find a finite open cover $\mathcal{U}_{O}$
of $L\circ O(K)$ and $0<\varepsilon_{O}<\frac{1}{2}$ such that $\sum_{U\in\mathcal{U}_{O}}\mathrm{diam}(U)^{t}\cdot(1+2\varepsilon_{O})^{t}<c$.
Let $\zeta_{O}=\min_{U\in\mathcal{U}_{O}}\mathrm{diam}(U)>0$ and
$\widehat{U}$ be the $\zeta_{O}\cdot\varepsilon_{O}$-neigbourhood
of $U$ for each $U\in\mathcal{U}_{O}$. We can find $\delta_{O}>0$
such that if $O_{2}\in\mathbb{O}_{d}$ and $\left\Vert O-O_{2}\right\Vert <\delta_{O}$
then $L\circ O_{2}(K)$ is contained in the $\zeta_{O}\cdot\varepsilon_{O}$-neighbourhood
of $L\circ O(K)$, hence $L\circ O_{2}(K)$ is covered by $\left\{ \widehat{U}:U\in\mathcal{U}_{O}\right\} $.
Then $\left\{ \widehat{U}:U\in\mathcal{U}_{O}\right\} $ is an open
cover of $L\circ O_{2}(K)$,
\[
\sum_{U\in\mathcal{U}_{O}}\mathrm{diam}(\widehat{U})^{t}\leq\sum_{U\in\mathcal{U}_{O}}\mathrm{diam}(U)^{t}\cdot(1+2\varepsilon_{O})^{t}<c
\]
and $\min_{U\in\mathcal{U}_{O}}\mathrm{diam}(\widehat{U})>\zeta_{O}$.

As $\mathcal{G}$ is compact, we can find finitely many orthogonal
transformations $O_{1},\ldots,O_{n}\in\mathcal{G}$ such that for
every $O\in\mathcal{G}$ there exists $i\in\left\{ 1,\ldots,n\right\} $
with $\left\Vert O_{i}-O\right\Vert <\delta_{O_{i}}$. Hence $\zeta=\min_{1\leq i\leq n}\zeta_{O_{i}}$
satisfies the statement.
\end{proof}

For $r\in\mathbb{R}$, $r>0$ and $H\subseteq\mathbb{R}^{d}$ we denote
the \textit{$r$-neigbourhood} of $H$ by $B(H,r)$, i.e. $B(H,r)=\left\{ x\in\mathbb{R}^{d}:\exists y\in H,\left\Vert x-y\right\Vert <r\right\} $.

\begin{prop}
\label{cor:Linim 0 lemakeppen}Let $\left\{ S_{i}\right\} _{i=1}^{m}$
be an SS-IFS with attractor $K$, $t=\dim_{H}(K)$ and $L:\mathbb{R}^{d}\longrightarrow\mathbb{R}^{d_{2}}$
be a linear map with $\mathrm{rank}(L)=l$. If $1\leq l<d$ and there
exists $M\in G_{d,l}$ such that the set $\left\{ O(M):O\in\mathcal{T}\right\} $
is dense in $G_{d,l}$ then $\mathcal{H}^{t}\left(L(K)\right)=0$.
\end{prop}

\noindent We first show that there exist two words $a$ and $b$ and
$O_{0}\in\mathbb{O}_{d}$ such that $L\circ O_{0}(K_{a})$ and $L\circ O_{0}(K_{b})$
have very large overlap. Then we use a variant of Proposition \ref{lem:Egyenlito disjoint lem T=00003Dinfty}
to show that, due to self-similarity, this remains valid at all scales.
Finally we conclude that due to these overlaps the measure must collapse.

\begin{proof}
It holds in general that $\mathcal{H}^{t}(H)=0$ if and only if $\mathcal{H}_{\infty}^{t}(H)=0$.
Hence it is enough to show that $\mathcal{H}_{\infty}^{t}\left(L(K)\right)=0$.
We can assume that $\mathcal{H}^{t}\left(K\right)>0$ otherwise the
statement is trivial. By Proposition \ref{prop:.implicitkov} $\mathcal{H}^{t}\left(K\right)<\infty$,
hence $K$ is a $t$-set. It follows that $\mathcal{H}^{t}\left(L(K)\right)<\infty$
and by Theorem \ref{thm: Egynlo lin Im} and Proposition \ref{lem:Lin neighbourhood lem}
$\mathcal{H}_{\infty}^{t}\left(L(K)\right)=\mathcal{H}^{t}\left(L(K)\right)=\mathcal{H}^{t}\left(L\circ O_{0}(K)\right)=\mathcal{H}_{\infty}^{t}\left(L\circ O_{0}(K)\right)$
for every $O_{0}\in\overline{\mathcal{T}}$. Let $\varepsilon>0$
be arbitrary, $\mathcal{G}=\overline{\mathcal{T}}$, $c=\mathcal{H}_{\infty}^{t}\left(L(K)\right)+\varepsilon$
and $\zeta>0$ be the $\zeta$ provided by Lemma \ref{lem: eta-cover lemma}.
We can find $\delta>0$ such that for every linear map $L_{2}:\mathbb{R}^{d}\longrightarrow\mathbb{R}^{d}$
such that $\left\Vert Id_{\mathbb{R}^{d}}-L_{2}\right\Vert <\delta$
we have that $L_{2}(K)\subseteq B\left(K,\varepsilon\zeta\right)$.
By Proposition \ref{lem:Egyenlito disjoint lem T=00003Dinfty} we
can find $\boldsymbol{\mathbf{i}}_{1},\boldsymbol{\mathbf{i}}_{2}\in\bigcup_{k=1}^{\infty}\mathcal{I}^{k}$
such that $K_{\boldsymbol{\mathbf{i}}_{1}}\cap K_{\boldsymbol{\mathbf{i}}_{2}}=\emptyset$
and $\left\Vert T_{\boldsymbol{\mathbf{i}}_{1}}-Id_{\mathbb{R}^{d}}\right\Vert <\frac{\delta}{4}$,
$\left\Vert T_{\boldsymbol{\mathbf{i}}_{2}}-Id_{\mathbb{R}^{d}}\right\Vert <\frac{\delta}{4}$.
Let $a=\boldsymbol{\mathbf{i}}_{1}*\boldsymbol{\mathbf{i}}_{2}$ and
$b=\boldsymbol{\mathbf{i}}_{2}*\boldsymbol{\mathbf{i}}_{1}$. Then
$\left\Vert T_{a}-Id_{\mathbb{R}^{d}}\right\Vert <\frac{\delta}{2}$,
$\left\Vert T_{b}-Id_{\mathbb{R}^{d}}\right\Vert <\frac{\delta}{2}$,
$K_{a}\cap K_{b}=\emptyset$ and $r_{a}=r_{b}$. Let $v=S_{b}(0)-S_{a}(0)$
and $O_{0}\in\overline{\mathcal{T}}$ such that $L\circ O_{0}(v)=0$.
We can choose such an $O_{0}$ by Lemma \ref{lem: Mag lemma}. We
can find $\delta_{2}>0$ such that if $\left\Vert L\circ O_{0}-L_{2}\right\Vert <\delta_{2}$
then $\left\Vert L_{2}(v)\right\Vert <r_{a}\varepsilon\zeta$.

We can apply Proposition \ref{lem:Sok reszes Egyenlito disjoint lem}
with $\min\left\{ \frac{\delta}{2},\frac{\delta_{2}}{\left\Vert L\right\Vert }\right\} $
replacing $\delta$, $a$ replacing $\boldsymbol{\mathbf{i}}_{1}$,
$b$ replacing $\boldsymbol{\mathbf{i}}_{2}$, $n=2$ and $O=Id_{\mathbb{R}^{d}}$
to obtain $\mathcal{I}_{\infty}\subseteq\bigcup_{k=1}^{\infty}\mathcal{I}^{k}$
such that $\left\Vert T_{\boldsymbol{\mathbf{i}}}-Id_{\mathbb{R}^{d}}\right\Vert <\min\left\{ \frac{\delta}{2},\frac{\delta_{2}}{\left\Vert L\right\Vert }\right\} $
for all $\boldsymbol{\mathbf{i}}\in\mathcal{I}_{\infty}$, with $\bigcup_{\boldsymbol{\mathbf{i}}\in\mathcal{I}_{\infty}}\left(K_{\boldsymbol{\mathbf{i}}*a}\cup K_{\boldsymbol{\mathbf{i}}*b}\right)$
a disjoint union and $\mathcal{H}^{t}\left(K\setminus\left(\bigcup_{\boldsymbol{\mathbf{i}}\in\mathcal{I}_{\infty}}\left(K_{\boldsymbol{\mathbf{i}}*a}\cup K_{\boldsymbol{\mathbf{i}}*b}\right)\right)\right)=0$.

So $\left\Vert T_{\boldsymbol{\mathbf{i}}}\circ T_{a}-Id_{\mathbb{R}^{d}}\right\Vert <\delta$
and $\left\Vert T_{\boldsymbol{\mathbf{i}}}\circ T_{b}-Id_{\mathbb{R}^{d}}\right\Vert <\delta$,
hence $T_{\boldsymbol{\mathbf{i}}}\circ T_{a}(K)\subseteq B(K,\varepsilon\zeta)$
and $T_{\boldsymbol{\mathbf{i}}}\circ T_{b}(K)\subseteq B(K,\varepsilon\zeta)$.
Thus
\[
r_{\boldsymbol{\mathbf{i}}}r_{a}T_{\boldsymbol{\mathbf{i}}}\circ T_{a}(K)\cup r_{\boldsymbol{\mathbf{i}}}r_{b}T_{\boldsymbol{\mathbf{i}}}\circ T_{b}(K)\subseteq B(r_{\boldsymbol{\mathbf{i}}}r_{a}K,r_{\boldsymbol{\mathbf{i}}}r_{a}\varepsilon\zeta)
\]
since $r_{a}=r_{b}$. Hence
\[
O_{0}\circ S_{\boldsymbol{\mathbf{i}}*a}(K)\subseteq B\left(r_{\boldsymbol{\mathbf{i}}}r_{a}O_{0}(K)+O_{0}\circ S_{\boldsymbol{\mathbf{i}}*a}(0),r_{\boldsymbol{\mathbf{i}}}r_{a}\varepsilon\zeta\right)
\]
and
\[
O_{0}\circ S_{\boldsymbol{\mathbf{i}}*b}(K)\subseteq B\left(r_{\boldsymbol{\mathbf{i}}}r_{a}O_{0}(K)+O_{0}\circ S_{\boldsymbol{\mathbf{i}}*b}(0),r_{\boldsymbol{\mathbf{i}}}r_{a}\varepsilon\zeta\right).
\]
Hence
\[
L\circ O_{0}\circ S_{\boldsymbol{\mathbf{i}}*a}(K)\subseteq B\left(r_{\boldsymbol{\mathbf{i}}}r_{a}L\circ O_{0}(K)+L\circ O_{0}\circ S_{\boldsymbol{\mathbf{i}}*a}(0),\left\Vert L\right\Vert r_{\boldsymbol{\mathbf{i}}}r_{a}\varepsilon\zeta\right)
\]
and
\[
L\circ O_{0}\circ S_{\boldsymbol{\mathbf{i}}*b}(K)\subseteq B\left(r_{\boldsymbol{\mathbf{i}}}r_{a}L\circ O_{0}(K)+L\circ O_{0}\circ S_{\boldsymbol{\mathbf{i}}*a}(0)+L\circ O_{0}\circ r_{\boldsymbol{\mathbf{i}}}T_{\boldsymbol{\mathbf{i}}}(v),\left\Vert L\right\Vert r_{\boldsymbol{\mathbf{i}}}r_{a}\varepsilon\zeta\right).
\]
By the choice of $\delta_{2}$ we have $\left\Vert L\circ O_{0}\circ r_{\boldsymbol{\mathbf{i}}}T_{\boldsymbol{\mathbf{i}}}(v)\right\Vert <r_{\boldsymbol{\mathbf{i}}}r_{a}\varepsilon\zeta$.
Hence
\[
L\circ O_{0}\circ S_{\boldsymbol{\mathbf{i}}*a}(K)\cup L\circ O_{0}\circ S_{\boldsymbol{\mathbf{i}}*b}(K)\subseteq B\left(r_{\boldsymbol{\mathbf{i}}}r_{a}L\circ O_{0}(K)+L\circ O_{0}\circ S_{\boldsymbol{\mathbf{i}}*a}(0),(\left\Vert L\right\Vert +1)r_{\boldsymbol{\mathbf{i}}}r_{a}\varepsilon\zeta\right).
\]

By the choice of $\zeta$ there exists an open cover $\mathcal{U}$
of $L\circ O_{0}(K)$ such that $\sum_{U\in\mathcal{U}}\mathrm{diam}(U)^{t}<\mathcal{H}_{\infty}^{t}\left(L(K)\right)+\varepsilon$
and $\min_{U\in\mathcal{U}}\mathrm{diam}(U)>\zeta$. Let $\widehat{U}=B\left(U,(\left\Vert L\right\Vert +1)\varepsilon\zeta\right)$
for each $U\in\mathcal{U}$ and $\mathcal{A}=\left\{ r_{\boldsymbol{\mathbf{i}}}r_{a}\widehat{U}+L\circ O_{0}\circ S_{\boldsymbol{\mathbf{i}}*a}(0):U\in\mathcal{U},\boldsymbol{\mathbf{i}}\in\mathcal{I}_{\infty}\right\} $.
Then $\mathcal{A}$ is an open cover of $L\circ O_{0}\left(\bigcup_{\boldsymbol{\mathbf{i}}\in\mathcal{I}_{\infty}}\left(K_{\boldsymbol{\mathbf{i}}*a}\cup K_{\boldsymbol{\mathbf{i}}*b}\right)\right)$.

Because $K$ is a $t$-set it follows in a similar way to Remark \ref{Rem: Mainlem EQ}
that $\sum_{\boldsymbol{\mathbf{i}}\in\mathcal{I}_{\infty}}r_{\boldsymbol{\mathbf{i}}}^{t}(r_{a}^{t}+r_{b}^{t})=1$,
hence $\sum_{\boldsymbol{\mathbf{i}}\in\mathcal{I}_{\infty}}r_{\boldsymbol{\mathbf{i}}}^{t}(r_{a}^{t})=\frac{1}{2}$
because $r_{a}=r_{b}$. Thus 
\[
\begin{aligned}\sum_{A\in\mathcal{A}}\mathrm{diam}(A)^{t} & \leq\sum_{U\in\mathcal{U}}\sum_{\boldsymbol{\mathbf{i}}\in\mathcal{I}_{\infty}}r_{\boldsymbol{\mathbf{i}}}^{t}r_{a}^{t}\left(\mathrm{diam}(U)+2(\left\Vert L\right\Vert +1)\varepsilon\zeta\right)^{t}\\
 & \leq\sum_{U\in\mathcal{U}}\sum_{\boldsymbol{\mathbf{i}}\in\mathcal{I}_{\infty}}r_{\boldsymbol{\mathbf{i}}}^{t}r_{a}^{t}\mathrm{diam}(U)^{t}\left(1+2(\left\Vert L\right\Vert +1)\varepsilon\right)^{t}\\
 & =\sum_{U\in\mathcal{U}}\frac{1}{2}\mathrm{diam}(U)^{t}\left(1+2(\left\Vert L\right\Vert +1)\varepsilon\right)^{t}\\
 & \leq\frac{1}{2}\left(\mathcal{H}_{\infty}^{t}\left(L(K)\right)+\varepsilon\right)\left(1+2(\left\Vert L\right\Vert +1)\varepsilon\right)^{t}.
\end{aligned}
\]
Because $\mathcal{H}^{t}\left(K\setminus\left(\bigcup_{\boldsymbol{\mathbf{i}}\in\mathcal{I}_{\infty}}\left(K_{\boldsymbol{\mathbf{i}}*a}\cup K_{\boldsymbol{\mathbf{i}}*b}\right)\right)\right)=0$
it follows that
\[
\mathcal{H}_{\infty}^{t}\left(L\circ O_{0}\left(K\setminus\left(\bigcup_{\boldsymbol{\mathbf{i}}\in\mathcal{I}_{\infty}}\left(K_{\boldsymbol{\mathbf{i}}*a}\cup K_{\boldsymbol{\mathbf{i}}*b}\right)\right)\right)\right)=0.
\]
Hence
\[
\mathcal{H}_{\infty}^{t}\left(L\circ O_{0}(K)\right)\leq\frac{1}{2}\left(\mathcal{H}_{\infty}^{t}\left(L(K)\right)+\varepsilon\right)\left(1+2(\left\Vert L\right\Vert +1)\varepsilon\right)^{t}.
\]
Since this is true for all $\varepsilon>0$ it follows that $\mathcal{H}_{\infty}^{t}\left(L(K)\right)=\mathcal{H}_{\infty}^{t}\left(L\circ O_{0}(K)\right)\leq\frac{1}{2}\cdot\mathcal{H}_{\infty}^{t}\left(L(K)\right)$.
Thus $\mathcal{H}_{\infty}^{t}\left(L(K)\right)=0$.
\end{proof}

\begin{cor}
\label{cor:Linim 0 lemmakeppen KOV}Let $\left\{ S_{i}\right\} _{i=1}^{m}$
be an SS-IFS with attractor $K$, let $t=\dim_{H}(K)$ and $L:\mathbb{R}^{d}\longrightarrow\mathbb{R}^{d_{2}}$
be a linear map with $\mathrm{rank}(L)\leq l$. If $1\leq l<d$ and
there exists $M\in G_{d,l}$ such that the set $\left\{ O(M):O\in\mathcal{T}\right\} $
is dense in $G_{d,l}$ then $\mathcal{H}^{t}\left(L(K)\right)=0$.
\end{cor}

\begin{proof}
If $L:\mathbb{R}^{d}\longrightarrow\mathbb{R}^{d_{2}}$ is a linear
map of rank $k$ and $k\leq l<d$ then $\dim\mathrm{Ker}(L)=d-k$.
Let $N\in G_{d,d-l}$ such that $N\subseteq\mathrm{Ker}(L)$. Then
$L=L\circ\Pi_{N^{\perp}}$. It follows from Proposition \ref{cor:Linim 0 lemakeppen}
that $\mathcal{H}^{t}\left(\Pi_{N^{\perp}}(K)\right)=0$. Hence $\mathcal{H}^{t}\left(L(K)\right)=0$.
\end{proof}

\begin{lem}
\label{lem: LIN-eta-cover lem}Let $K\subseteq\mathbb{R}^{d}$ be
a compact set and $c,M>0$ be constants such that $\mathcal{H}_{\infty}^{t}\left(L(K)\right)<c$
for every linear map $L:\mathbb{R}^{d}\longrightarrow\mathbb{R}^{d_{2}}$
with $\left\Vert L\right\Vert \leq M$. Then there exists $\zeta>0$
such that for every linear map $L:\mathbb{R}^{d}\longrightarrow\mathbb{R}^{d_{2}}$
with $\left\Vert L\right\Vert \leq M$ there exists a finite open
cover $\mathcal{U}$ of $L(K)$ such that $\sum_{U\in\mathcal{U}}\mathrm{diam}(U)^{t}<c$
and $\min_{U\in\mathcal{U}}\mathrm{diam}(U)>\zeta$.
\end{lem}

Lemma \ref{lem: LIN-eta-cover lem} can be proven similarly to Lemma
\ref{lem: eta-cover lemma} due the fact that the unit ball of the
set of linear maps between two finite dimensional Euclidean spaces
is compact.\\

\noindent \textit{Proof of Theorem \ref{cor:Dense group lin image 0}.}
We can assume that $\mathcal{H}^{t}\left(K\right)>0$ otherwise the
statement is trivial since $g$ is a Lipschitz map. By Proposition
\ref{prop:.implicitkov} $\mathcal{H}^{t}\left(K\right)<\infty$ and
hence $K$ is a $t$-set. Let $\varepsilon>0$ be fixed.

Let $x_{0}\in K$ be arbitrary. It follows from Corollary \ref{cor:Linim 0 lemmakeppen KOV}
that
\[
\mathcal{H}_{\infty}^{t}\left(g'(S_{\boldsymbol{\mathbf{i}}}(x_{0}))\circ T_{\boldsymbol{\mathbf{i}}}(K)\right)=\mathcal{H}^{t}\left(g'(S_{\boldsymbol{\mathbf{i}}}(x_{0}))\circ T_{\boldsymbol{\mathbf{i}}}(K)\right)=0.
\]
Let $c=\varepsilon$ and $M=\sup\left\{ \left\Vert g'(x)\right\Vert :x\in K\right\} <\infty$,
then let $\zeta>0$ be the $\zeta$ provided by Lemma \ref{lem: LIN-eta-cover lem}.
Hence for every $\boldsymbol{\mathbf{i}}\in\bigcup_{k=1}^{\infty}\mathcal{I}^{k}$
there exists a finite open cover $\mathcal{U}_{\boldsymbol{\mathbf{i}}}$
of $g'(S_{\boldsymbol{\mathbf{i}}}(x_{0}))\circ T_{\boldsymbol{\mathbf{i}}}(K)$
such that
\begin{equation}
\sum_{U\in\mathcal{U}_{\boldsymbol{\mathbf{i}}}}\mathrm{diam}(U)^{t}<\varepsilon\label{eq:diffim_diam<epsilon}
\end{equation}
and $\min_{U\in\mathcal{U}_{\boldsymbol{\mathbf{i}}}}\mathrm{diam}(U)>\zeta$.

From the continuous differentiability of $g$ and the compactness
of $K$ it follows that we can find $\delta>0$ such that
\[
\left\Vert g(y)-g(x)-g'(x)\cdot\left(y-x\right)\right\Vert <\zeta\cdot\left\Vert y-x\right\Vert 
\]
for $x,y\in K$ such that $\left\Vert y-x\right\Vert <\delta$ (see
\cite[Exercise 7(c).3]{Burkill-II Analysis}). By Lemma \ref{lem:disjointing lem}
and Remark \ref{rem: disjointing rem} we can find $\mathcal{J}\subseteq\bigcup_{k=1}^{\infty}\mathcal{I}^{k}$
such that $K_{\boldsymbol{\mathbf{i}}}\cap K_{\boldsymbol{\mathbf{j}}}=\emptyset$
for $\boldsymbol{\mathbf{i}},\boldsymbol{\mathbf{j}}\in\mathcal{J}$,
$\boldsymbol{\mathbf{i}}\neq\boldsymbol{\mathbf{j}}$ and $\mathrm{diam}(K_{\boldsymbol{\mathbf{i}}})<\delta$
for every $\boldsymbol{\mathbf{i}}\in\mathcal{J}$ and
\begin{equation}
\mathcal{H}^{t}\left(K\setminus\left(\bigcup_{\boldsymbol{\mathbf{i}}\in\mathcal{J}}K_{\boldsymbol{\mathbf{i}}}\right)\right)=0.\label{eq:exhaust_diff image}
\end{equation}
Similarly to Remark \ref{Rem: Mainlem EQ} it follows that
\begin{equation}
\sum_{\boldsymbol{\mathbf{i}}\in\mathcal{J}}r_{\boldsymbol{\mathbf{i}}}^{t}=1.\label{eq:diffim_simdim}
\end{equation}

For every $y\in K$ we have that $\left\Vert S_{\boldsymbol{\mathbf{i}}}(y)-S_{\boldsymbol{\mathbf{i}}}(x_{0})\right\Vert \leq\mathrm{diam}(K_{\boldsymbol{\mathbf{i}}})<\delta$
and hence by the choice of $\delta$ it follows that
\[
\left\Vert g\left(S_{\boldsymbol{\mathbf{i}}}(y)\right)-g\left(S_{\boldsymbol{\mathbf{i}}}(x_{0})\right)-g'\left(S_{\boldsymbol{\mathbf{i}}}(x_{0})\right)\cdot\left(S_{\boldsymbol{\mathbf{i}}}(y)-S_{\boldsymbol{\mathbf{i}}}(x_{0})\right)\right\Vert <\zeta\cdot\left\Vert S_{\boldsymbol{\mathbf{i}}}(y)-S_{\boldsymbol{\mathbf{i}}}(x_{0})\right\Vert \leq\zeta r_{\boldsymbol{\mathbf{i}}}\mathrm{diam}(K).
\]
Thus
\begin{equation}
g\left(S_{\boldsymbol{\mathbf{i}}}(y)\right)\in B\left(g'\left(S_{\boldsymbol{\mathbf{i}}}(x_{0})\right)\left(S_{\boldsymbol{\mathbf{i}}}(K)\right)+g\left(S_{\boldsymbol{\mathbf{i}}}(x_{0})\right)-g'\left(S_{\boldsymbol{\mathbf{i}}}(x_{0})\right)\cdot\left(S_{\boldsymbol{\mathbf{i}}}(x_{0})\right),\zeta r_{\boldsymbol{\mathbf{i}}}\mathrm{diam}(K)\right).\label{eq:eta-neigbour_Lin}
\end{equation}
Since $\mathcal{U}_{\boldsymbol{\mathbf{i}}}$ is an open cover of
$g'(S_{\boldsymbol{\mathbf{i}}}(x_{0}))\circ T_{\boldsymbol{\mathbf{i}}}(K)$
it follows from (\ref{eq:eta-neigbour_Lin}) that
\[
\left\{ B\left(r_{\boldsymbol{\mathbf{i}}}U+g'\left(S_{\boldsymbol{\mathbf{i}}}(x_{0})\right)\cdot g\left(S_{\boldsymbol{\mathbf{i}}}(0)\right)+g\left(S_{\boldsymbol{\mathbf{i}}}(x_{0})\right)-g'\left(S_{\boldsymbol{\mathbf{i}}}(x_{0})\right)\cdot\left(S_{\boldsymbol{\mathbf{i}}}(x_{0})\right),\zeta r_{\boldsymbol{\mathbf{i}}}\mathrm{diam}(K)\right):U\in\mathcal{U}_{\boldsymbol{\mathbf{i}}}\right\} 
\]
is an open cover of $g\left(S_{\boldsymbol{\mathbf{i}}}(K)\right)$.
We have that
\[
\mathrm{diam}\left(B\left(r_{\boldsymbol{\mathbf{i}}}U+g'\left(S_{\boldsymbol{\mathbf{i}}}(x_{0})\right)\cdot g\left(S_{\boldsymbol{\mathbf{i}}}(0)\right)+g\left(S_{\boldsymbol{\mathbf{i}}}(x_{0})\right)-g'\left(S_{\boldsymbol{\mathbf{i}}}(x_{0})\right)\cdot\left(S_{\boldsymbol{\mathbf{i}}}(x_{0})\right),\zeta r_{\boldsymbol{\mathbf{i}}}\mathrm{diam}(K)\right)\right)
\]
\begin{equation}
\leq\mathrm{diam}\left(B\left(r_{\boldsymbol{\mathbf{i}}}\cdot U,\zeta r_{\boldsymbol{\mathbf{i}}}\mathrm{diam}(K)\right)\right)\leq r_{\boldsymbol{\mathbf{i}}}\left(\mathrm{diam}(U)+2\zeta\mathrm{diam}(K)\right)\label{eq:diam in diff ima}
\end{equation}
\[
\leq r_{\boldsymbol{\mathbf{i}}}\mathrm{diam}(U)\left(1+2\mathrm{diam}(K)\right).
\]
Since $g$ is a Lipschitz map it follows from (\ref{eq:exhaust_diff image})
that
\[
\mathcal{H}_{\infty}^{t}\left(g\left(K\setminus\left(\bigcup_{\boldsymbol{\mathbf{i}}\in\mathcal{J}}K_{\boldsymbol{\mathbf{i}}}\right)\right)\right)=0.
\]
Hence by (\ref{eq:diam in diff ima})
\[
\begin{aligned}\mathcal{H}_{\infty}^{t}\left(g(K)\right) & \leq\sum_{\boldsymbol{\mathbf{i}}\in\mathcal{J}}\sum_{U\in\mathcal{U}_{\boldsymbol{\mathbf{i}}}}r_{\boldsymbol{\mathbf{i}}}^{t}\mathrm{diam}(U)^{t}\left(1+2\mathrm{diam}(K)\right)^{t}\\
 & \leq\left(1+2\mathrm{diam}(K)\right)^{t}\sum_{\boldsymbol{\mathbf{i}}\in\mathcal{J}}r_{\boldsymbol{\mathbf{i}}}^{t}\sum_{U\in\mathcal{U}_{\boldsymbol{\mathbf{i}}}}\mathrm{diam}(U)^{t}<\left(1+2\mathrm{diam}(K)\right)^{t}\varepsilon
\end{aligned}
\]
where we used (\ref{eq:diffim_diam<epsilon}) and (\ref{eq:diffim_simdim}).
Since this is true for every $\varepsilon>0$ it follows that $\mathcal{H}_{\infty}^{t}\left(g(K)\right)=0$
and hence $\mathcal{H}^{t}\left(g(K)\right)=0$.$\hfill\square$

\section{Transformation groups of dense orbits and Hausdorff dimension}

In this section we prove Theorem \ref{thm:Dense group dim conserv}
and Corollary \ref{cor:Dense diff dim conserv COR}. First we show
Proposition \ref{lem:dim approx lem} and then we derive Theorem \ref{thm:Dense group dim conserv}
from \cite[Corollary 1.7]{Hochman-Schmerkin-local entropy} and Proposition
\ref{lem:dim approx lem}. Finally we conclude Corollary \ref{cor:Dense diff dim conserv COR}
from Theorem \ref{thm:Dense group dim conserv}.

\begin{lem}
\label{lem: fixpoint lem}Let $S_{1}:\mathbb{R}^{d}\longrightarrow\mathbb{R}^{d}$
and $S_{2}:\mathbb{R}^{d}\longrightarrow\mathbb{R}^{d}$ be contracting
similarities with no common fixed point. Then the similarities $S_{1}^{n}\circ S_{2}$
have different fixed points for all $n\in\mathbb{N}$.
\end{lem}

\begin{proof}
By Banach's fixed point theorem every contracting similarity $S:\mathbb{R}^{d}\longrightarrow\mathbb{R}^{d}$
has a unique fixed point. Assume for a contradiction that there exist
$x\in\mathbb{R}^{d}$, $n\in\mathbb{N}$ and a positive integer $k$
such that that $S_{1}^{n}\circ S_{2}(x)=x$ and $S_{1}^{k}\circ S_{1}^{n}\circ S_{2}(x)=x$.
Then $S_{1}^{-k}(x)=S_{1}^{n}\circ S_{2}(x)=x$. It follows that the
unique fixed point of $S_{1}$ is $x$. But then $S_{2}(x)=S_{1}^{-n}(x)=x$
contradicting that $S_{1}$ and $S_{2}$ have no common fixed point.
\end{proof}

\begin{lem}
\label{lem:FP change}Let $S_{1},\ldots,S_{m}:\mathbb{R}^{d}\longrightarrow\mathbb{R}^{d}$
be contracting similarities $(m\geq2)$ such that $S_{1}$ and $S_{2}$
have no common fixed point. Then there exist $F_{1},\ldots,F_{m}:\mathbb{R}^{d}\longrightarrow\mathbb{R}^{d}$
such that $F_{1}=S_{1}$, $F_{2}=S_{2}$, for each $i\in\left\{ 3,\ldots,m\right\} $
either $F_{i}=S_{1}^{k_{i}}\circ S_{i}$ or $F_{i}=S_{2}^{k_{i}}\circ S_{i}$
for some $k_{i}\in\mathbb{N}$, and $F_{i}$ and $F_{j}$ have no
common fixed point for all $i,j\in\left\{ 1,\ldots,m\right\} $, $i\neq j$.
\end{lem}

\begin{proof}
We prove this by induction on $m$. If $m=2$ then it is trivial.
Let $m>2$. Then by the inductive assumption we can find such a system
$F_{1},\ldots,F_{m-1}$ that satisfies the conclusion for $S_{1},\ldots,S_{m-1}$.
The unique fixed point of $S_{m}$ is either not the fixed point of
$S_{1}$ or not the fixed point of $S_{2}$. Without the loss of generality
we can assume that $S_{m}$ and $S_{1}$ have no common fixed points.
Then by Lemma \ref{lem: fixpoint lem} there exists $k_{m}\in\mathbb{N}$
such that the fixed point of $S_{1}^{k_{m}}\circ S_{m}$ is different
from the fixed points of $F_{1},\ldots,F_{m-1}$. If we set $F_{m}=S_{1}^{k_{m}}\circ S_{m}$
then $F_{1},\ldots,F_{m}$ satisfies the conclusion.
\end{proof}

\noindent Now we are ready to prove Proposition \textit{\ref{lem:dim approx lem}}.
The proof consists of two steps. First we find a collection of words
$\boldsymbol{\mathbf{j}}_{i}\in\bigcup_{k=1}^{\infty}\mathcal{I}^{k}$
such that the group generated by $T_{\boldsymbol{\mathbf{j}}_{i}}$
is dense in $\mathcal{T}$ and the $S_{\boldsymbol{\mathbf{j}}_{i}}(K)$
are disjoint. At this point we do not care about the dimension. Then
we add another finite set of maps to the new SS-IFS so that the strong
separation condition still holds and the dimension becomes arbitrarily
close to that of $K$.\\

\noindent \textit{Proof of Proposition \ref{lem:dim approx lem}.}
Since $K$ has at least two points there exist $i,j\in\mathcal{I}$
such that $S_{i}$ and $S_{j}$ have no common fixed point, otherwise
the common fixed point would be the attractor. Without the loss of
generality we can assume that $i=1$ and $j=2$.

It follows from Lemma \ref{lem:FP change} that there exist $\boldsymbol{\mathbf{i}}_{1},\ldots,\boldsymbol{\mathbf{i}}_{m}\in\bigcup_{k=1}^{\infty}\mathcal{I}^{k}$
such that $S_{\boldsymbol{\mathbf{i}}_{i}}$ and $S_{\boldsymbol{\mathbf{i}}_{j}}$
have no common fixed point for all $i,j\in\mathcal{I}$, $i\neq j$,
$\boldsymbol{\mathbf{i}}_{1}=1$, $\boldsymbol{\mathbf{i}}_{2}=2$
and the group generated by $T_{\boldsymbol{\mathbf{i}}_{1}},\ldots,T_{\boldsymbol{\mathbf{i}}_{m}}$
is $\mathcal{T}$. Let $x_{i}$ be the unique fixed point of $S_{\boldsymbol{\mathbf{i}}_{i}}$
for all $i\in\mathcal{I}$. Let $d_{\mathrm{min}}=\min\left\{ \left\Vert x_{i}-x_{j}\right\Vert :i,j\in\mathcal{I},i\neq j\right\} >0$,
$r_{\mathrm{max}}=\max\left\{ r_{i}:i\in\mathcal{I}\right\} <1$ and
$N\in\mathbb{N}$ such that $r_{\mathrm{max}}^{N}\cdot\mathrm{diam}(K)<\frac{d_{\mathrm{min}}}{2}$.
Then $S_{\boldsymbol{\mathbf{i}}_{i}}^{k_{i}}(K)\cap S_{\boldsymbol{\mathbf{i}}_{j}}^{k_{j}}(K)=\emptyset$
for all $i,j\in\mathcal{I}$, $i\neq j$, $k_{i},k_{j}\in\mathbb{N}$,
$k_{i},k_{j}\geq N$.

By Proposition \ref{lem:Jordanos lem} for all $i\in\mathcal{I}$
we can find $k_{i}\in N$, $k_{i}\geq N$ such that the group generated
by $T_{\boldsymbol{\mathbf{i}}_{i}}^{k_{i}}$ is dense in the group
generated by $T_{\boldsymbol{\mathbf{i}}_{i}}$. It follows that the
group generated by $T_{\boldsymbol{\mathbf{i}}_{1}}^{k_{1}},\ldots,T_{\boldsymbol{\mathbf{i}}_{m}}^{k_{m}}$
is dense in $\mathcal{T}$ and $S_{\boldsymbol{\mathbf{i}}_{i}}^{k_{i}}(K)\cap S_{\boldsymbol{\mathbf{i}}_{j}}^{k_{j}}(K)=\emptyset$
for all $i,j\in\mathcal{I}$, $i\neq j$. Let $\widehat{S_{i}}=S_{\boldsymbol{\mathbf{i}}_{i}}^{k_{i}}$
for all $i\in\mathcal{I}$.

Let $F=\bigcup_{i\in\mathcal{I}}S_{\boldsymbol{\mathbf{i}}_{i}}^{k_{i}}(K)$.
If $K=F$ then $\left\{ \widehat{S_{i}}\right\} _{i=1}^{m}$ satisfies
the SSC with attractor $\widehat{K}=K$ and the proof is complete.
So we can assume that $F\subsetneq K$. Let $\boldsymbol{\mathbf{j}}\in\bigcup_{k=1}^{\infty}\mathcal{I}^{k}$
be such that $K_{\boldsymbol{\mathbf{j}}}\cap F=\emptyset$. Let $t=\dim_{H}K=\dim_{H}K_{\boldsymbol{\mathbf{j}}}$.
Since $K$ has at least two points it follows that $K$ has infinitely
many points but by Proposition \ref{prop:.implicitkov} $\mathcal{H}^{t}(K)<\infty$,
thus $t>0$ and hence without the loss of generality we can assume
that $t>\varepsilon>0$. Since $\mathcal{H}^{t-\frac{\varepsilon}{2}}(K_{\boldsymbol{\mathbf{j}}})=\infty$
we can find $\delta>0$ such that for any $3\delta$-cover $\mathcal{U}$
of $K_{\boldsymbol{\mathbf{j}}}$ we have that $\sum_{U\in\mathcal{U}}\mathrm{diam}(U)^{t-\frac{\varepsilon}{2}}>1$.
Let $r_{\mathrm{min}}=\min\left\{ r_{i}:i\in\mathcal{I}\right\} <1$
and let $\mathcal{J}=\left\{ \boldsymbol{\mathbf{i}}\in\bigcup_{k=1}^{\infty}\mathcal{I}^{k}:K_{\boldsymbol{\mathbf{i}}}\subseteq K_{\boldsymbol{\mathbf{j}}},r_{\mathrm{min}}\delta\leq\mathrm{diam}(K_{\boldsymbol{\mathbf{i}}})<\delta\right\} $.
Then $\left\{ K_{\boldsymbol{\mathbf{i}}}:\boldsymbol{\mathbf{i}}\in\mathcal{J}\right\} $
is a cover of $K_{\boldsymbol{\mathbf{j}}}$. Let $\boldsymbol{\mathbf{j}}_{1},\ldots,\boldsymbol{\mathbf{j}}_{n}\in\mathcal{J}$
be such that $K_{\boldsymbol{\mathbf{j}}_{1}},\ldots,K_{\boldsymbol{\mathbf{j}}_{n}}$
is a maximal pairwise disjoint sub-collection of $\left\{ K_{\boldsymbol{\mathbf{i}}}:\boldsymbol{\mathbf{i}}\in\mathcal{J}\right\} $.
Let $U_{j}$ be the $\delta$-neighbourhood of $K_{\boldsymbol{\mathbf{j}}_{j}}$
for $j\in\left\{ 1,\ldots,n\right\} $. By the maximality $\left\{ U_{j}:j\in\left\{ 1,\ldots,n\right\} \right\} $
is a $3\delta$-cover of $K_{\boldsymbol{\mathbf{j}}}$. Hence by
the choice of $\delta$
\[
\sum_{j=1}^{n}\left(3\delta\right)^{t-\frac{\varepsilon}{2}}\geq\sum_{j=1}^{n}\left(\mathrm{diam}(U_{j})\right)^{t-\frac{\varepsilon}{2}}>1.
\]
It follows that $n\geq\left(3\delta\right)^{-\left(t-\frac{\varepsilon}{2}\right)}$.
Let $K_{0}$ be the attractor of the SS-IFS $\left\{ S_{\boldsymbol{\mathbf{j}}_{j}}\right\} _{j=1}^{n}$.
Then $K_{0}\subseteq K$, the SS-IFS $\left\{ S_{\boldsymbol{\mathbf{j}}_{j}}\right\} _{j=1}^{n}$
satisfies the SSC and
\[
\dim_{H}K_{0}\geq\frac{\log(\frac{1}{n})}{\log(\frac{r_{\mathrm{min}}\cdot\delta}{\mathrm{diam}(K)})}\geq\frac{-\left(t-\frac{\varepsilon}{2}\right)\cdot\log(3)-\left(t-\frac{\varepsilon}{2}\right)\cdot\log(\delta)}{\log(\mathrm{diam}(K))-\log(r_{\mathrm{min}})-\log(\delta)}
\]
because the similarity dimension of $\left\{ S_{\boldsymbol{\mathbf{j}}_{j}}\right\} _{j=1}^{n}$
is $\dim_{H}K_{0}$ by Proposition \ref{prop: OSC schief prop}. So,
by choosing $\delta$ small enough, $\dim_{H}K_{0}>t-\varepsilon$.
Let $\widehat{m}=m+n$, $\widehat{S_{m+j}}=S_{\boldsymbol{\mathbf{j}}_{j}}$
for all $j\in\left\{ 1,\ldots,n\right\} $ and $\widehat{K}$ be the
attractor of the SS-IFS $\left\{ \widehat{S_{i}}\right\} _{i=1}^{\widehat{m}}$.
Then the transformation group $\widehat{\mathcal{T}}$ of $\left\{ \widehat{S_{i}}\right\} _{i=1}^{\widehat{m}}$
is dense in $\mathcal{T}$, $K_{0}\subseteq\widehat{K}\subseteq K$,
$\dim_{H}K-\varepsilon<\dim_{H}K_{0}\leq\dim_{H}\widehat{K}$ and
$\left\{ \widehat{S_{i}}\right\} _{i=1}^{\widehat{m}}$ satisfies
the SSC.$\hfill\square$\\

A similar argument to the last step of the proof of Proposition \ref{lem:dim approx lem}
was used in the proof of \cite[Theorem 2]{Peres-Schmerkin Resonance between Cantor sets}.\\

\noindent \textit{Proof of Theorem \ref{thm:Dense group dim conserv}.}
The upper bound $\dim_{H}\left(g(K)\right)\leq\min\left\{ t,l\right\} $
follows since $g$ is a Lipschitz map on $K$.

First assume that $\mathrm{rank}(g'(x))=l$ holds for every $x\in U$.
By Proposition \ref{lem:dim approx lem}, for all $\varepsilon>0$
there exists an SS-IFS $\left\{ \widehat{S_{i}}\right\} _{i=1}^{\widehat{m}}$
that satisfies the SSC with attractor $\widehat{K}$ such that $\widehat{K}\subseteq K$,
$\dim_{H}K-\varepsilon<\dim_{H}\widehat{K}$ and for the transformation
group $\widehat{\mathcal{T}}$ of $\left\{ \widehat{S_{i}}\right\} _{i=1}^{\widehat{m}}$
we have that $\left\{ O(M):O\in\widehat{\mathcal{T}}\right\} $ is
dense in $G_{d,l}$. By \cite[Corollary 1.7]{Hochman-Schmerkin-local entropy}
$\dim_{H}\left(g(\widehat{K})\right)=\min\left\{ \dim_{H}\widehat{K},l\right\} $.
Hence
\[
\dim_{H}\left(g(K)\right)\geq\dim_{H}\left(g(\widehat{K})\right)=\min\left\{ \dim_{H}\widehat{K},l\right\} \geq\min\left\{ \dim_{H}K-\varepsilon,l\right\} .
\]
So $\dim_{H}\left(g(K)\right)\geq\min\left\{ t-\varepsilon,l\right\} $
for all $\varepsilon>0$ and hence $\dim_{H}\left(g(K)\right)=\min\left\{ t,l\right\} $.

In the general case there exists $x\in K$ such that $\mathrm{rank}(g'(x))=l$
it follows that there exists an open neighbourhood $V$ of $x$ such
that $\mathrm{rank}(g'(y))=l$ for every $y\in V$. For large enough
$k$ there exists $\boldsymbol{\mathbf{i}}\in\mathcal{I}^{k}$ such
that $K_{\boldsymbol{\mathbf{i}}}\subseteq V$. Then $K_{\boldsymbol{\mathbf{i}}}$
is the attractor of the SS-IFS $\left\{ S_{\boldsymbol{\mathbf{i}}}\circ S_{j}\circ S_{\boldsymbol{\mathbf{i}}}^{-1}\right\} _{j=1}^{m}$
and $\left\{ O(M):O\in\mathcal{T}_{\boldsymbol{\mathbf{i}}}\right\} $
is dense in $G_{d,l}$ where $\mathcal{T}_{\boldsymbol{\mathbf{i}}}$
is the transformation group of $\left\{ S_{\boldsymbol{\mathbf{i}}}\circ S_{j}\circ S_{\boldsymbol{\mathbf{i}}}^{-1}\right\} _{j=1}^{m}$.
Thus we can assume that $\mathrm{rank}(g'(x))=l$ holds for every
$x\in U$.$\hfill\square$\\

\noindent \textit{Proof of Corollary \ref{cor:Dense diff dim conserv COR}.}
Let $g(y)=\left(g_{1}(y),\ldots,g_{d_{2}}(y)\right)$ and set an arbitrary
point $x\in K$. Since $\mathrm{rank}(g'(x))=l$ it follows that there
are $l$ coordinate indices $1\leq j_{1}\leq\ldots\leq j_{l}\leq d_{2}$
such that the vectors $g'_{j_{1}}(x),\ldots,g'_{j_{l}}(x)$ are linearly
independent. Let $P:\mathbb{R}^{d_{2}}\longrightarrow\mathbb{R}^{l}$
be the projection $P(y)=\left(y_{j_{1}},\ldots,y_{j_{l}}\right)$
and $f:U\longrightarrow\mathbb{R}^{l}$ be $f(y)=P\circ g(y)$. Note
that $P$ and hence $f$ may depend on $x$. Then the conditions of
Theorem \ref{thm:Dense group dim conserv} are satisfied for $f$
in place of $g$. Thus $\dim_{H}P\circ g(K)=\min\left\{ t,l\right\} $
and hence $\dim_{H}g(K)\geq\min\left\{ t,l\right\} $.

The upper bound in case of \textit{i)} follows from \cite[Theorem 3.4.3]{Federer book}.
The upper bound in case of \textit{ii)} follows since $g$ is a Lipschitz
map on $K$ and hence $\dim_{H}g(K)\leq t=\min\left\{ t,l\right\} $.

\section{Examples and questions\label{sec:Examples-and-questions}}

In this section we raise some open questions and provide examples.

If the Hausdorff dimension coincides with the similarity dimension
and $\left|\mathcal{T}\right|<\infty$ then by Theorem \ref{thm:dimension drop SS}
there must be at least one projection where the dimension drops. The
following example is very well-known and shows that it is possible
to have a projection of positive measure when $\left|\mathcal{T}\right|<\infty$.

\begin{example}
\label{Ex: poz proj sierp}For $0<t\leq1$ the $t$-dimensional Sierpinski
triangle is the attractor of the SS-IFS that contains three homotheties
which map an equilateral triangle into itself fixing the corners with
similarity ratio $r=3^{-1/t}$. Then $\left|\mathcal{T}\right|=1$
and $\mathcal{H}^{t}\left(\Pi_{M}(K)\right)>0$ where $M$ is a line
paralell to one of the sides of the triangle.
\end{example}

\begin{question}
Is it true that if $\left|\mathcal{T}\right|<\infty$ and $t\leq l<d$
then we can always find $l$-dimensional subspaces $M_{1}$ and $M_{2}$
such that $\mathcal{H}^{t}\left(\Pi_{M_{1}}(K)\right)>0$ and $\dim_{H}\left(\Pi_{M_{2}}(K)\right)<t$?
\end{question}

Theorem \ref{cor:Dense group lin image 0} shows that if $\left\{ O(M):O\in\mathcal{T}\right\} $
is dense in $G_{d,l}$ for some $M\in G_{d,l}$ then every projection
is of zero measure, on the other hand Theorem \ref{thm:Dense group dim conserv}
shows that there is no projection where the dimension drops. Example
\ref{Ex: vegyes proj} shows that $\left|\mathcal{T}\right|=\infty$
is not enough to imply either of these results.
\begin{example}
\label{Ex: vegyes proj}There exists a self-similar set $K$ with
$t=\dim_{H}\left(K\right)$ such that $\left|\mathcal{T}\right|=\infty$
and there exist three different orthogonal projections $P_{1},P_{2},P_{3}$
of $K$ onto lines with the following properties: $t=\dim_{H}\left(P_{1}\right)$
and $\mathcal{H}^{t}\left(P_{1}\right)=0$, $\mathcal{H}^{t}\left(P_{2}\right)>0$
and $\dim_{H}\left(P_{3}\right)<t$. Let $T_{1}:\mathbb{R}^{2}\longrightarrow\mathbb{R}^{2}$
be a rotation around the origin by angle $\alpha\cdot\pi$ for some
$\alpha\notin\mathbb{Q}$ and let $T:\mathbb{R}^{2}\times\mathbb{R}^{2}\longrightarrow\mathbb{R}^{2}\times\mathbb{R}^{2}$
be defined as $T(x,y)=(T_{1}(x),y)$ for $x,y\in\mathbb{R}^{2}$.
Let $r\leq\frac{1}{3}$ and $v_{1,i}\in\mathbb{R}^{2}$ for $i=1,2,3$
be such that the SS-IFS $\left\{ r\cdot T_{1}(x)+v_{1,i}\right\} _{i=1}^{3}$
satisfies the SSC with attractor $K_{1}$. Let $v_{2,i}\in\mathbb{R}^{2}$
for $i=1,2,3$ be such that the attractor of the SS-IFS $\left\{ r\cdot Id_{\mathbb{R}^{2}}(x)+v_{2,i}\right\} _{i=1}^{3}$
is the $\frac{\log(3)}{\log(r^{-1})}$-dimensional Sierpinski triangle
$K_{2}$. Set $v_{i}=(v_{1,i},v_{2,i})\in\mathbb{R}^{2}\times\mathbb{R}^{2}$,
$S_{i}(z)=r\cdot T(z)+v_{i}$ for $z\in\mathbb{R}^{2}\times\mathbb{R}^{2}$,
$i=1,2,3$ and let $K$ be the attractor of the SS-IFS $\left\{ S_{i}\right\} _{i=1}^{3}$.
Then $\left\{ S_{i}\right\} _{i=1}^{3}$ satisfies the SSC, hence
$t=\dim_{H}K=\dim_{H}K_{1}=\dim_{H}K_{2}=\frac{\log(3)}{\log(r^{-1})}$.
Let $M_{1}=\mathbb{R}^{2}\times(0,0)$, let $L_{1}\subseteq M_{1}$
be any line, $M_{2}=(0,0)\times\mathbb{R}^{2}$ and $L_{2}=(0,0)\times\mathbb{R}\times(0)$.
One can show that $\Pi_{M_{1}}(K)=K_{1}$, thus $P_{1}=\Pi_{L_{1}}(K)=\Pi_{L_{1}}\circ\Pi_{M_{1}}(K)=\Pi_{L_{1}}(K_{1})$
and hence by Theorem \ref{thm:Dense group dim conserv} and Theorem
\ref{cor:Dense group lin image 0} $\dim_{H}\left(\Pi_{L_{1}}(K)\right)=t$
and $\mathcal{H}^{t}\left(\Pi_{L_{1}}(K)\right)=0$. On the other
hand $\Pi_{M_{2}}(K)=K_{2}$, thus $P_{2}=\Pi_{L_{2}}(K)=\Pi_{L_{2}}\circ\Pi_{M_{2}}(K)=\Pi_{L_{2}}(K_{2})$
and hence $\mathcal{H}^{t}\left(\Pi_{L_{2}}(K)\right)>0$. Finally
by Theorem \ref{thm:dimension drop SS} there exists a line $L_{3}\subseteq M_{2}$
such that $\dim_{H}\left(\Pi_{L_{3}}(K_{2})\right)<t$ and hence $\dim_{H}\left(\Pi_{L_{3}}(K)\right)<t$.
The transformation group $\mathcal{T}$ of $\left\{ S_{i}\right\} _{i=1}^{3}$
is infinite, but $\left\{ O(L):O\in\mathcal{T}\right\} $ is not dense
in $G_{4,1}$ for any $L\in G_{4,1}$ and $K$ is not contained in
any affine hyperplane.
\end{example}

Is there a counterexample to the converse of Proposition \ref{cor:Linim 0 lemakeppen}
and Theorem \ref{thm:Dense group dim conserv}?

\begin{question}
Is there an example of an SS-IFS such that $\mathcal{T}$ is infinite,
$\dim_{H}\left(\Pi_{N}(K)\right)=t$ and $\mathcal{H}^{t}\left(\Pi_{N}(K)\right)=0$
for every $N\in G_{d,l}$, but $\left\{ O(M):O\in\mathcal{T}\right\} $
is not dense in $G_{d,l}$ for any $M\in G_{d,l}$ and $\mathcal{H}^{t}(K)>0$?
\end{question}

The following example shows that we cannot replace $g$ with a Lipschitz
function in Theorem \ref{thm:Dense group dim conserv} and Theorem
\ref{cor:Dense group lin image 0}.

\begin{example}
\label{cannot replace}Let $t\leq1$ and $\left\{ S_{i}\right\} _{i=1}^{3}$
be an SS-IFS in $\mathbb{R}^{2}$ such that $S_{1}$ and $S_{2}$
are two maps from the usual SS-IFS of the $t$-dimensional Sierpinski
triangle and we slightly modify the orthogonal part of the third map
so that $T_{3}$ is a rotation of a small angle $\alpha\cdot\pi$
for some $\alpha\notin\mathbb{Q}$. Let $K$ be the attractor of $\left\{ S_{i}\right\} _{i=1}^{3}$
and $\widehat{K}$ be the $t$-dimensional Sierpinski triangle. Then
one can show that the natural bijection $f$ between $K$ and $\widehat{K}$
is a bi-Lipschitz function. Then the assumptions of Theorem \ref{cor:Dense group lin image 0}
holds for $\left\{ S_{i}\right\} _{i=1}^{3}$ and $l=1$ but there
exist lines $M_{1}$ and $M_{2}$ such that $\mathcal{H}^{t}\left(\Pi_{M_{1}}(f(K))\right)>0$
and $\dim_{H}\left(\Pi_{M_{2}}(f(K))\right)<t$.
\end{example}

The two following examples show that the assumption $\mathcal{H}^{t}(K)>0$
where $t=\dim_{H}K$ is weaker than the OSC.

\begin{example}
\label{Ex: nonOSC-line}There exists a self-similar set $\widehat{K}\subseteq\mathbb{R}$
such that no SS-IFS with attractor $\widehat{K}$ satisfies the OSC
but $0<\mathcal{H}^{t}(\widehat{K})<\infty$ where $t=\dim_{H}\widehat{K}$.

Let $0<r<\frac{1}{3}$ and $g=\frac{1-3r}{2}$. We first define an
SS-IFS as follows (see Figure 1): $S_{1}(x)=r\cdot x$, $S_{2}(x)=r\cdot x+r+g$
and $S_{3}(x)=r\cdot x+2r+2g$. We denote by $K$ the attractor of
$\left\{ S_{i}\right\} _{i=1}^{3}$. Since $\left\{ S_{i}\right\} _{i=1}^{3}$
satisfies the OSC it follows that $0<\mathcal{H}^{t}(K)<\infty$ where
$t=\dim_{H}K$. The set $\widehat{K}=K\setminus S_{3}(K)=S_{1}(K)\cup S_{2}(K)$
is also a self similar set, namely it is the attractor of an SS-IFS
containing the following four maps: $\widehat{S_{1}}(x)=S_{1}(x)$,
$\widehat{S_{2}}(x)=S_{1}(x)+r(r+g)$, $\widehat{S_{3}}(x)=S_{2}(x)$
and $\widehat{S_{4}}(x)=S_{2}(x)+r(r+g)$. We have that $0<\mathcal{H}^{t}(\widehat{K})<\infty$.

Let $F(x)=a\cdot x+b$ a contractive similarity such that $F(\widehat{K})\subseteq\widehat{K}$.
We show that $a=r^{n}$ for some positive integer $n$. We call the
length of the longest bounded component of the complement of a compact
set the \textit{largest gap}.

First assume that $r\leq a<1$. The largest gap of $\widehat{K}$
is $g$ and the largest gap of $F(\widehat{K})$ is $ag<g$. The distance
between $S_{1}(K)$ and $S_{2}(K)$ is $g$ hence either $F(\widehat{K})\subseteq S_{1}(K)$
or $F(\widehat{K})\subseteq S_{2}(K)$. For simplicity assume that
$F(\widehat{K})\subseteq S_{1}(K)$, the proof goes similarly in the
case $F(\widehat{K})\subseteq S_{2}(K)$. The largest gap of $F\circ S_{1}(K)$
is $arg<rg$. The smallest distance between the sets $S_{1}\circ S_{1}(K)$,
$S_{1}\circ S_{2}(K)$ and $S_{1}\circ S_{3}(K)$ is $rg$. Hence
either $F\circ S_{1}(K)\subseteq S_{1}\circ S_{1}(K)$ or $F\circ S_{1}(K)\subseteq S_{1}\circ S_{2}(K)$
or $F\circ S_{1}(K)\subseteq S_{1}\circ S_{3}(K)$. Thus $ar\leq rr$
and so $a\leq r$. Since we assumed $r\leq a<1$ it follows that $a=r$.

Now assume that $r^{n}\leq a<r^{n-1}$ for some positive integer $n$.
As above we can show that $F(\widehat{K})\subseteq S_{\boldsymbol{\mathbf{i}}}(K)$
for some $\boldsymbol{\mathbf{i}}\in\left\{ 1,2,3\right\} ^{n}$ and
$F\circ S_{1}(K)\subseteq S_{\boldsymbol{\mathbf{i}}}\circ S_{j}(K)$
for some $j\in\left\{ 1,2,3\right\} $. Hence $a=r^{n}$.

Since $F(\widehat{K})\subseteq S_{\boldsymbol{\mathbf{i}}}(K)$ for
some $\boldsymbol{\mathbf{i}}\in\left\{ 1,2,3\right\} ^{n}$ and $a=r^{n}$
it follows that either
\begin{equation}
F(\widehat{K})=S_{\boldsymbol{\mathbf{i}}}\circ S_{1}(K)\cup S_{\boldsymbol{\mathbf{i}}}\circ S_{2}(K)\,\mathrm{or}\,F(\widehat{K})=S_{\boldsymbol{\mathbf{i}}}\circ S_{2}(K)\cup S_{\boldsymbol{\mathbf{i}}}\circ S_{3}(K)\label{eq:nonosc exeq}
\end{equation}
because the largest gap of $F(\widehat{K})$ and one of the largest
gaps of $S_{\boldsymbol{\mathbf{i}}}(K)$ must coincide.

Let $\left\{ F_{i}\right\} _{i=1}^{m}$ be an SS-IFS with attractor
$\widehat{K}$. Without the loss of generality we can assume that
the similarity ratio $r^{n}$ of $F_{1}(x)$ is the smallest of the
similarity ratios of the maps $F_{i}$. By (\ref{eq:nonosc exeq})
and the minimality of $r^{n}$ there exists $j\in\left\{ 1,\ldots,m\right\} $
such that $S_{\boldsymbol{\mathbf{i}}}(K)\setminus F_{1}(\widehat{K})\subseteq F_{j}(\widehat{K})$
and either $F_{1}(\widehat{K})\cap F_{j}(\widehat{K})=F_{1}(\widehat{K})$
or $F_{1}(\widehat{K})\cap F_{j}(\widehat{K})=S_{\boldsymbol{\mathbf{i}}}\circ S_{2}(K)$.
Thus $\mathcal{H}^{t}\left(F_{1}(\widehat{K})\cap F_{j}(\widehat{K})\right)>0$
and so $\left\{ F_{i}\right\} _{i=1}^{m}$ cannot satisfy the OSC.
\end{example}

\includegraphics[scale=0.5]{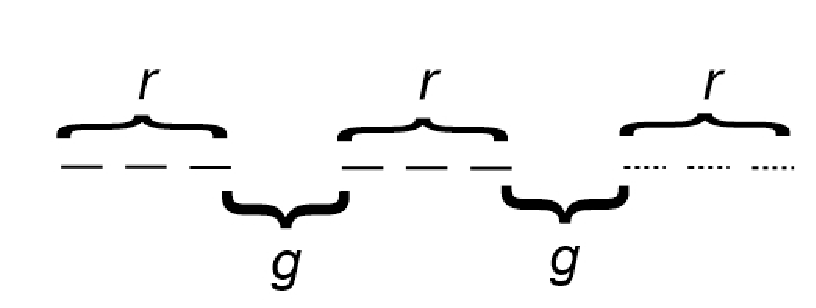}~~~~~~~~\includegraphics[scale=0.3]{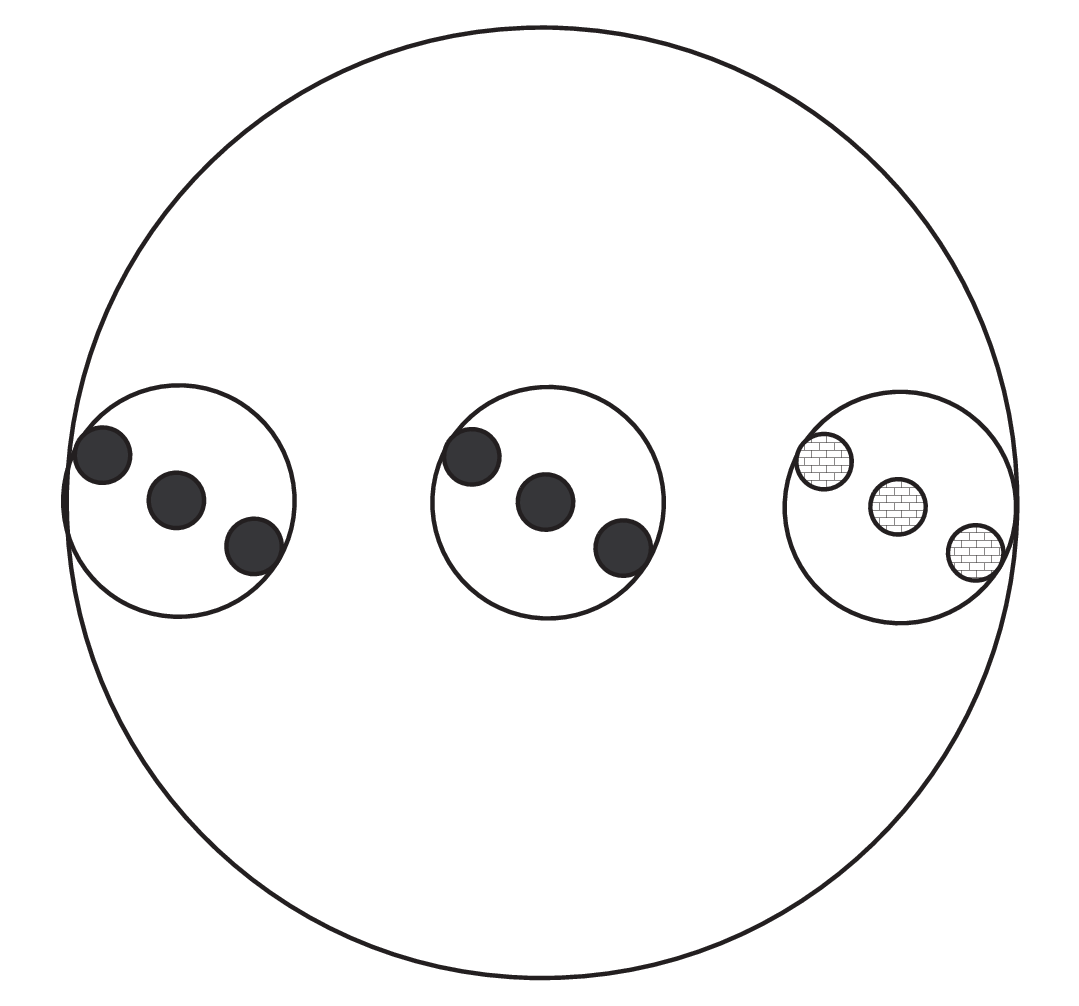}

~~~~~~~~~~~~~~~~~~~~~Figure 1.~~~~~~~~~~~~~~~~~~~~~~~~~~~~~~~~~~~~~~~~Figure
2.

\begin{example}
\label{Ex: nonOSC-Plane}Let $T$ be a rotation around the origin
by angle $\alpha\in[0,2\pi)$. There exists a self-similar set $\widehat{K}\subseteq\mathbb{R}^{2}$
such that no SS-IFS with attractor $\widehat{K}$ satisfies the OSC
but $0<\mathcal{H}^{t}(\widehat{K})<\infty$ where $t=\dim_{H}\widehat{K}$
and there exists an SS-IFS with attractor $\widehat{K}$ such that
the transformation group of the SS-IFS is generated by $T$.

Let $0<r<\frac{1}{3}$ and $g=1-3r$. We define an SS-IFS in $\mathbb{R}^{2}$
as follows (see Figure 2): $S_{1}(x)=rT(x)+(-g-2r,0)$, $S_{2}(x)=rT(x)$
and $S_{3}(x)=rT(x)+(g+2r,0)$. We denote by $K$ the attractor of
$\left\{ S_{i}\right\} _{i=1}^{3}$. Since $\left\{ S_{i}\right\} _{i=1}^{3}$
satisfies the OSC it follows that $0<\mathcal{H}^{t}(K)<\infty$ where
$t=\dim_{H}K$. The set $\widehat{K}=K\setminus S_{3}(K)=S_{1}(K)\cup S_{2}(K)$
is also a self-similar set, namely it is the attractor of an SS-IFS
containing the following four maps: $\widehat{S_{1}}(x)=S_{1}(x)$,
$\widehat{S_{2}}(x)=S_{1}\left(x+(g+2r,0)\right)$, $\widehat{S_{3}}(x)=S_{2}(x)$
and $\widehat{S_{4}}(x)=S_{2}\left(x+(g+2r,0)\right)$. We have that
$0<\mathcal{H}^{t}(\widehat{K})<\infty$ and the transformation group
of $\left\{ \widehat{S_{i}}\right\} _{i=1}^{4}$ is generated by $T$.

We can show that there is no SS-IFS with attractor $\widehat{K}$
that satisfies the OSC via a similar argument to the argument in Example
\ref{Ex: nonOSC-line} with the difference that the largest gap of
$K$ and $\widehat{K}$ will be replaced by the smallest distance
between $S_{1}(K)$ and $S_{2}(K)$. We note that this distance is
greater than $g$. 
\end{example}

\begin{rem}
We note that both in Example \ref{Ex: nonOSC-line} and Example \ref{Ex: nonOSC-Plane}
the semigroup generated by $\left\{ \widehat{S_{i}}\right\} _{i=1}^{4}$
is not free. Hence after iteration and deleting repetitions one can
reduce the similarity dimension of the SS-IFS. It is not hard to see
that we can find an SS-IFS with attractor $\widehat{K}$ of similarity
dimension arbitrarily close to $t$ but we cannot find an SS-IFS with
attractor $\widehat{K}$ of similarity dimension $t$ because of Proposition
\ref{prop: OSC schief prop}.
\end{rem}

\begin{center}
$\mathbf{Acknowledgements}$
\par\end{center}

The author was supported by an EPSRC doctoral training grant and thanks
Kenneth Falconer, Jonathan Fraser and the anonymous referee for many
valuable suggestions on improving the exposition of the article.

\end{document}